\documentclass[hidelinks,onefignum,onetabnum]{siamart251216}

\headers{Linear Decision Tree Policies for Integer Linear Programs}{T. Guyard, C. Oliveira, M. Schiffer, E. Uchoa, and T. Vidal}

\title{
    Linear Decision Tree Policies for \\ Integer Linear Programs%
    \thanks{Submitted to the editors: May 8, 2026.
    }
}

\author{
    \phantom{AAA} 
    Théo Guyard\thanks{SCALE-AI Chair in Data-Driven Supply Chains, Polytechnique Montréal, Canada.}\and%
    Cleber Oliveira\thanks{Tecgraf Institute, Pontifical Catholic University of Rio de Janeiro, Brazil.}\and%
    Maximilian Schiffer\thanks{School of Management \& MDSI, Technical University of Munich, Germany.}\and%
    \phantom{AAA} 
    Eduardo Uchoa\thanks{Fluminense Federal University, Brazil.}\and%
    Thibaut Vidal\footnotemark[2]%
}

\usepackage[setbb]{kmath}
\usepackage[utf8]{inputenc}
\usepackage[T1]{fontenc}

\usepackage{ifthen}

\usepackage{url} 
\usepackage[acronym]{glossaries}
\glsdisablehyper
\usepackage{cleveref}
\usepackage{lipsum}
\usepackage{enumitem}
\usepackage{comment}
\usepackage[sort,square,numbers]{natbib}

\usepackage{amsmath}
\usepackage{amsopn}
\usepackage{mathtools}

\usepackage{xcolor}
\usepackage{array}
\usepackage{rotating}
\usepackage[dvipsnames]{xcolor}
\usepackage{graphicx}
\usepackage{tikz}
\usetikzlibrary{decorations.pathmorphing,arrows.meta,positioning,shapes.multipart,calc,patterns,patterns.meta}
\usepackage{pgfplots}
\pgfplotsset{compat=newest}
\usepackage{framed}
\usepackage{booktabs}
\usepackage{siunitx}
\usepackage{subcaption}
\usepackage{pgfplotstable}
\usepackage{xstring}
\usetikzlibrary{intersections}
\usepgfplotslibrary{fillbetween}
\usepgfplotslibrary{groupplots}
\usetikzlibrary{matrix}
\usepackage[noend]{algpseudocode}
\usepackage{mdframed}
\usepackage{listings}
\lstdefinestyle{CStyle}{
    language=C,
    basicstyle=\ttfamily\small,
    keywordstyle=\color{blue}\bfseries,
    stringstyle=\color{teal},
    commentstyle=\color{gray}\itshape,
    numbers=left,
    numberstyle=\tiny\color{gray},
    stepnumber=1,
    numbersep=8pt,
    backgroundcolor=\color{gray!5},
    showspaces=false,
    showstringspaces=false,
    showtabs=false,
    frame=single,
    rulecolor=\color{gray!30},
    tabsize=4,
    captionpos=b,
    breaklines=true,
    breakatwhitespace=true,
    escapeinside={(*@}{@*)},
    morekeywords={uint, uint32_t, uint64_t}, 
}

\newsiamthm{assumption}{Assumption}
\newsiamthm{property}{Property}
\newsiamthm{remark}{Remark}
\newsiamthm{problem}{Problem}

\crefname{assumption}{assumption}{assumptions}
\crefname{property}{property}{properties}
\crefname{remark}{remark}{remarks}
\crefname{problem}{problem}{problems}

\crefname{subsection}{Section}{Sections}
\Crefname{subsection}{Section}{Sections}

\newcommand{\pset}{\mathcal{X}}

\newcommand{\pdim}{n}

\newcommand{\pvletter}{x}
\newcommand{\cvletter}{c}
\newcommand{\pv}{\mathbf{\pvletter}}
\newcommand{\cv}{\mathbf{\cvletter}}
\newcommand{\pvi}[1]{\pvletter_{#1}}

\newcommand{\policy}{\pi}
\newcommand{\policies}{\Pi}

\newcommand{\dualfaces}{\mathcal{F}}
\newcommand{\dualface}{F}
\newcommand{\divider}{h}
\newcommand{\dividers}{\mathcal{H}}
\newcommand{\costdomain}{\mathcal{C}}
\newcommand{\scalingfunc}{\phi}
\newcommand{\treecond}{M}
\newcommand{\idxentry}{i}

\newcommand{\nodesymbol}{N}

\newcommand{\nodelt}{H_{<}}

\newcommand{\nodegt}{H_{>}}
\newcommand{\nodetuple}{(\nodelt,\nodegt)}
\newcommand{\noderoot}{\nodesymbol_{0}}
\newcommand{\noderegion}{\mathcal{R}}

\newcommand{\childlt}[2]{#1_{<#2}}

\newcommand{\childgt}[2]{#1_{>#2}}
\newcommand{\treedepth}{D}

\newcommand{\nodeheight}{\Delta}
\newcommand{\nodedepth}{d}
\newcommand{\discrepfunc}[1]{\mu_{#1}}
\newcommand{\sampled}[2]{#1^{#2}}
\newcommand{\hyperplanenumber}{\kappa}

\newcommand{\true}{\textsc{True}}
\newcommand{\false}{\textsc{False}}
\newcommand{\usat}{\mathbf{u}}

\newcommand{\sdim}{d}
\newcommand{\knpweight}{\mathbf{w}}

\newcommand{\knpcapacity}{W}

\DeclareMathOperator*{\argmin}{argmin}
\DeclareMathOperator*{\argmax}{argmax}

\DeclareMathOperator{\convhull}{conv}

\DeclarePairedDelimiter{\ceil}{\lceil}{\rceil}

\newcommand{\1}{\mathbf{1}}

\newcommand{\bigO}{\mathcal{O}}
\newcommand{\card}[1]{|#1|}

\newcommand{\norm}[2]{\|#1\|_#2}
\newcommand{\opt}[1]{#1^{\star}}

\newcommand{\transpose}[1]{#1^{\mathrm{T}}}

\newacronym{dp}{DP}{dynamic programming}
\newacronym{ilp}{ILP}{integer linear problem}
\newacronym{ldt}{LDT}{linear decision tree}
\newacronym{lp}{LP}{linear program}
\newacronym{nns}{NNS}{nearest neighbor search}
\newacronym{ram}{RAM}{random access memory}

\begin{document}

\maketitle

\begin{abstract}
We study optimal decision policies, represented as linear decision trees, for integer linear programs with a fixed feasible set and varying cost vectors. Once synthesized for a given feasible set, they return an optimal solution for any queried cost vector through a sequence of linear tests. We show that there exists a policy performing this operation in a polynomial number of arithmetic operations in the worst case. In contrast, deciding whether there exists an exact policy with a prescribed maximum number of leaves is $\Sigma_2^p$-complete. Alongside these theoretical results, we develop a practical construction framework to synthesize policies within a specific subclass of linear decision trees. Our computational experiments show that, although policy synthesis can be time-intensive, it allows one to retrieve optimal solutions orders of magnitude faster than classical and specialized solution methods on repeated queries. Overall, this paradigm provides a different perspective on the solution of integer linear programs and offers a principled offline-online approach for repeated optimization.
\end{abstract}



\vspace*{0.7cm}

\section{Introduction}
\label{sec:intro}

We focus on \glspl{ilp} of the form
\begin{equation}
    \max_{\pv \in \pset} \ \transpose{\cv}\pv
    \tag{$P$}
    \label{prob:prob}
\end{equation}
where $\cv \in \costdomain$ is a parameter vector within a domain $\costdomain \subseteq \kR^{\pdim}$ characterizing the objective function, and $\pset \subseteq \kZ^{\pdim}$ is a finite integer feasible set. Numerous such ILPs arise from applications across operations research, finance, and network design, among others~\citep{petropoulos2024operational}. The corresponding problems are often NP-hard. However, over several decades of research, a broad range of exact solution methods has emerged. Among the most successful approaches are those that exploit the structural properties of specific families of feasible sets. For instance, considerable effort has been devoted to deriving strong relaxations \citep{cornuejols1985traveling}, suitable decomposition schemes~\citep{desrosiers2005primer}, and exact polyhedral descriptions of the feasible set~\citep{Christof1996}. These insights have fueled significant improvements in the performance of modern solvers, enabling the solution of large instances \citep{Koch2022}.

\subsection*{Policies for Integer Linear Programs}
Despite these advances, state-of-the-art exact algorithms are typically not designed for efficient re-optimization when the parameter vector $\cv \in \costdomain$ varies. They need to be restarted from scratch and do not enable the exploitation of a fixed feasible-set structure across multiple instances. This observation motivates a fundamental algorithmic task: the synthesis of efficient \emph{decision policies} rather than solution methods for single instances, formalized as follows.\\

\begin{mdframed}
    \noindent
    Given a fixed feasible set $\pset \subseteq \kZ^{\pdim}$, a cost domain $\costdomain \subseteq \kR^{\pdim}$, and a family $\policies$ of policies $\policy : \costdomain \to \pset$, solve the problem
    \begin{equation}
    \label{eq:policy-design}
    \begin{array}{rlll}
        \textstyle
        \opt{\policy} \in &\argmin_{\policy \in \policies} & \ \ \Phi(\policy) \\
        & \text{subject to} 
        & \textstyle \ \ \policy(\cv) \in \argmax_{\pv \in \pset} \ \transpose{\cv}\pv
        & \forall \cv \in \costdomain
    \end{array}
    \end{equation}
    for some measure $\Phi : \policies \to \kR_+$ characterizing the complexity of evaluating a policy for a given cost vector.
\end{mdframed}
~

\noindent
From this perspective, a policy $\policy$ corresponds to a mapping from any cost vector in some target domain $\costdomain \subseteq \kR^{\pdim}$ to an optimal solution to the corresponding Problem~\eqref{prob:prob}. It is associated with a given feasible set $\pset \subseteq \kZ^{\pdim}$, which can be pre-processed during a \emph{construction} step. This operation may be computationally intensive since it is performed only once and can therefore be regarded as an offline cost. The family $\Pi$ reflects the nature of this construction operation and how the policy mapping is encoded. Once constructed, the policy can be \emph{evaluated} to retrieve an optimal solution corresponding to a given cost vector in Problem~\eqref{prob:prob}. The measure $\Phi$ reflects structural characteristics of the policy related to the complexity of its evaluation, which becomes the primary concern when several instances with varying cost vectors are solved.

The construction of decision policies raises fundamental questions: Can policies be constructed generically, regardless of the structure of the feasible set? Can general complexity guarantees be established on their evaluation? Can this paradigm yield practically meaningful algorithms for online optimization settings, where a series of problems with different parameter vectors must be solved? The purpose of this paper is to explore new answers to these questions.

\subsection*{Linear Decision Tree Policies}
In this paper, we focus on the family $\Pi$ of policies that encode a given feasible set $\pset \subseteq \kZ^{\pdim}$ into an \gls{ldt} during an offline construction step. Once such a policy has been synthesized, Problem~\eqref{prob:prob} is solved for a queried cost vector $\cv \in \costdomain$ by traversing the associated \gls{ldt}. Starting from the root, a branching decision is made at each node based on a linear test $\transpose{\mathbf{a}}\cv + b \odot 0$ with $\odot \in \{<,=,>\}$ for some $\mathbf{a} \in \kR^{\pdim}$ and $b \in \kR$. Eventually, the traversal reaches a leaf returning an optimal solution associated with this cost vector in Problem~\eqref{prob:prob}. If the tree has depth $\treedepth$, this query requires $\bigO(\pdim\treedepth)$ arithmetic operations in the worst case. This family of policies, therefore, provides both an interpretable representation and a natural notion of query complexity. We emphasize, however, that this is a non-uniform computational model: for each feasible set, a suitable \gls{ldt} must first be constructed and stored offline. Moreover, the complexity measure accounts only for the arithmetic operations performed during tree traversal and does not capture the cost of constructing and representing a potentially very large \gls{ldt} structure.

Some previous work has provided complexity analyses for \gls{ldt} policies, mostly from an existential perspective. In particular, Meyer auf der Heide \citep{meyer1984polynomial} highlighted that there exist \gls{ldt} policies of polynomial depth with respect to the problem dimension for specific classes of \glspl{ilp}, most notably the knapsack problem. This result uses geometric arguments based on hyperplane arrangements in the cost-vector space. Building on this research, Kolinek \citep{kolinek1987polynomial} demonstrated that for the traveling salesman, shortest path, and knapsack problems, there always exists an \gls{ldt} policy with depth $\treedepth \simeq 2\pdim^4\log_2(\pdim) + 2\pdim^4\log_2(\treecond) + \bigO(\pdim^3)$ for some $\treecond \in \kN$. Yet, no general framework covering arbitrary problems has been proposed, nor have concrete construction procedures been designed and empirically tested.

\subsection*{Contributions} 
In this paper, we pave the way toward generic \gls{ldt} policies for \glspl{ilp} and concrete construction strategies.
Our contributions are as follows:
\begin{itemize}
    \item We show that there exist \gls{ldt} policies that can be queried in a polynomial number of arithmetic operations for any~$\pset~\subseteq~\kZ^{\pdim}$ such that each $\pv \in \pset$ can be encoded in a polynomial number of bits. Our result thereby extends prior problem-specific results to a fairly general framework. In contrast, we show that deciding whether an exact \gls{ldt} policy with a prescribed maximum number of leaves exists is $\Sigma_2^p$-complete.
    \item We develop a systematic methodology to synthesize \gls{ldt} policies for \glspl{ilp}. Our approach seeks minimum-depth structures within a restricted class of \glspl{ldt}, in which linear tests are drawn from hyperplanes of the normal fan associated with the feasible set at hand.
    \item We show that the resulting \gls{ldt} policies benefit from an evaluation complexity smaller than that of the simplex method applied over the convex hull of the feasible set, even when initialized at the optimal solution. This highlights an intriguing perspective on the complexity of \glspl{ilp}. We also identify a notable connection with \glsdesc{nns} when the feasible set is binary.
    \item We conduct extensive computational experiments to assess the practical potential of this approach on small problem instances. In particular, we study both the construction and evaluation phases across several \gls{ilp} families. Although their synthesis can be computationally demanding, \gls{ldt} policies substantially outperform all benchmark methods in evaluation time.
\end{itemize}
\Cref{sec:existence,sec:dynamic} establish the existence and perform a practical synthesis of the proposed \gls{ldt} policies. \Cref{sec:numerics} contains numerical analyses of their construction and evaluation steps. \Cref{sec:related-works} provides an overview of related works, and \Cref{sec:conclusion} concludes our exposition with future research perspectives.

\section{Existence of Linear Decision Tree Policies}
\label{sec:existence}

We first investigate the existence of \gls{ldt} policies associated with generic \glspl{ilp}. Building on the notion of the normal fan associated with the feasible set and results from computational geometry, we show that an \gls{ldt} policy of polynomial depth always exists under a mild assumption. We then study the hardness of constructing such a policy. This approach also highlights a connection with \glsdesc{nns} problems in the pure binary case.

\subsection{Optimality Regions and Normal Fan}
\label{sec:existence:normal-fan}

For any $\pv \in \pset$, define the cone
\begin{equation}
    \label{eq:normal-fan-cone}
    F(\pv) = \kset{\cv \in \kR^{\pdim}}{\transpose{\cv}\pv \geq \transpose{\cv}\tilde{\pv} \;\; \forall \tilde{\pv} \in \pset}
\end{equation}
corresponding to all cost vectors for which this particular feasible solution is optimal in Problem~\eqref{prob:prob}, and denote by $\mathcal{V}(\pset)$ the set of vertices in $\convhull(\pset)$. Then, the family $\dualfaces(\pset) = \{F(\pv)\}_{\pv \in \mathcal{V}(\pset)}$ forms the \emph{normal fan} \citep[Chap.~7.1]{ziegler2012lectures} associated with $\convhull(\pset)$\footnote{For simplicity, we refer to $\dualfaces(\pset)$ as the normal fan associated with $\pset$.}.
By construction, we observe that evaluating a policy as defined in~\eqref{eq:policy-design} amounts to identifying the cones in this normal fan that contain the queried cost vector $\cv \in \kR^{\pdim}$.
This observation can be formalized as follows.

\begin{property}
    \label{prop:df-policy}
    Given the normal fan $\dualfaces(\pset) = \{F(\pv)\}_{\pv \in \mathcal{V}(\pset)}$ associated with some feasible set $\pset \subseteq \kZ^{\pdim}$, the policy
    \begin{equation}
        \label{eq:policy-voronoi}
        \policy(\cv) \in \kset{\pv \in \mathcal{V}(\pset)}{\cv \in F(\pv)}
    \end{equation}
    retrieves an optimal solution to Problem~\eqref{prob:prob} for any queried cost vector $\cv \in \costdomain$.
\end{property}

Note that the set in the right-hand side of \eqref{eq:policy-voronoi} contains all optimal solutions to Problem~\eqref{prob:prob}. However, only one is typically needed in practice. It can be selected using tie-breaking rules when multiple optima are present. Moreover, directly enumerating all cones in $\dualfaces(\pset)$ to find one that contains the queried cost vector is generally intractable. Indeed, the number of such cones, i.e., possible optimal solutions in Problem~\eqref{prob:prob}, generally grows exponentially with the problem dimension in classical \gls{ilp} optimization settings. In the sequel, we investigate how offline preprocessing effort can be invested in encoding the policy structure to facilitate its evaluation.

\subsection{Policies with Polynomial Evaluation Complexity}
\label{sec:existence:polynomial}

A suitable representation of the policy defined in \Cref{prop:df-policy} arises from hyperplanes separating adjacent cones in the normal fan $\dualfaces(\pset)$. Specifically, two distinct cones $\dualface(\pv)  \subseteq \kR^{\pdim}$ and $\dualface(\tilde{\pv}) \subseteq \kR^{\pdim}$ are said to be adjacent if $\dim(\dualface(\pv) \cap \dualface(\tilde{\pv})) = \pdim - 1$. In this case, they share a common boundary\footnote{All cones in $\dualfaces(\pset)$ share the linear space $\operatorname{lin}(\operatorname{aff}(\pset))^\perp$. Cost vectors in this subspace induce a constant objective over $\pset$ and can be handled separately, or the analysis may be carried on the quotient space obtained after factoring out this linear space.} corresponding to a hyperplane characterized by
\begin{equation}
    \divider_{\pv,\tilde{\pv}}(\cv) = \transpose{(\pv-\tilde{\pv})}\cv
    ,
\end{equation}
referred to as a \emph{divider}. Assessing whether $\divider_{\pv,\tilde{\pv}}(\cv) < 0$ or $\divider_{\pv,\tilde{\pv}}(\cv) > 0$ then determines which of the two associated feasible solutions in $\pset$ achieves the larger objective value for the queried cost vector in Problem~\eqref{prob:prob}. The case $\divider_{\pv,\tilde{\pv}}(\cv) = 0$ indicates that both achieve the same objective value.
Therefore, knowing the signs of all dividers in
\begin{equation}
    \label{eq:dividers}
    \dividers(\pset) = \kset{\divider_{\pv,\tilde{\pv}}}{(\pv,\tilde{\pv}) \in \mathcal{V}(\pset) \times \mathcal{V}(\pset), \ \pv \neq \tilde{\pv}, \ \dim(\dualface(\pv) \cap \dualface(\tilde{\pv})) = \pdim-1}
\end{equation}
when evaluated at the queried $\cv \in \costdomain$ is sufficient to recover an optimal solution to Problem~\eqref{prob:prob}. Notably, each divider is homogeneous (i.e., has no intercept) and is characterized by an integer-valued normal vector since $\pset \subseteq \kZ^{\pdim}$. This sign retrieval task can therefore be reduced to the so-called integer linear decision problem arising in computational geometry. Building on the prior work of Kolinek \citep{kolinek1987polynomial} in this vein, we obtain the following result.

\begin{property}
    \label{prop:query-complexity}
    Let $\pset \subseteq \kZ^{\pdim}$ and define $\treecond = \max_{\divider_{\pv,\tilde{\pv}} \in \dividers(\pset)}\norm{\pv-\tilde{\pv}}{\infty}$. Then, the policy given in \Cref{prop:df-policy} can be evaluated through a ternary \gls{ldt} of depth
    \begin{equation}
        \treedepth \simeq 2\pdim^4\log_2(\pdim) + 2\pdim^4\log_2(\treecond) + \bigO(\pdim^3)
    \end{equation}
    where branching operations are linear tests performed on the input cost vector $\cv \in \costdomain$.
\end{property}
\begin{proof}
    Querying the policy defined in \Cref{prop:df-policy} can be performed by assessing whether $\divider_{\pv,\tilde{\pv}}(\cv) < 0$, $\divider_{\pv,\tilde{\pv}}(\cv) = 0$, or $\divider_{\pv,\tilde{\pv}}(\cv) > 0$ for each $\divider_{\pv,\tilde{\pv}} \in \dividers(\pset)$. This operation corresponds to an integer linear decision problem with respect to the set of integer-valued normal vectors $\{\pv - \tilde{\pv}\}_{\divider_{\pv,\tilde{\pv}} \in \dividers(\pset)} \subseteq \kZ^{\pdim}$, as defined in \citep{kolinek1987polynomial}. Invoking the main theorem of \citep[Sec.~3]{kolinek1987polynomial} shows that this task can be carried out through a ternary \gls{ldt} of depth $\treedepth \simeq 2\pdim^4 \log_2(\pdim) + 2\pdim^4 \log_2(\treecond) + \bigO(\pdim^3)$ with $\treecond = \max_{\divider_{\pv,\tilde{\pv}} \in \dividers(\pset)}\norm{\pv - \tilde{\pv}}{\infty}$.
\end{proof}

Traversing an \gls{ldt} requires a number of arithmetic operations that is polynomial in its depth and in the dimension of the linear tests used for branching operations. The depth of the policy in \Cref{prop:query-complexity} is then polynomial in $\pdim$ whenever $\log_2(M)$ is polynomially bounded in $\pdim$. Translated to Problem~\eqref{prob:prob}, this condition requires that any feasible solution $\pv \in \pset$ can be encoded in a polynomial number of bits with respect to the dimension $\pdim$. Note that this assumption is met in most \gls{ilp} optimization settings. This leads to the following consequence.\\

\begin{mdframed}
    \noindent
    For any feasible set $\pset \subseteq \kZ^{\pdim}$ such that all solutions $\pv \in \pset$ can be encoded in a polynomial number of bits, there exists a decision policy for Problem~\eqref{prob:prob} represented as an \gls{ldt} whose query requires only a polynomial number of arithmetic operations with respect to $\pdim$.
\end{mdframed}
~

\noindent
This result is existential and shows that efficient policies for \glspl{ilp} exist in the \gls{ldt} query model, independently of the structure of the feasible set. A natural question, addressed in the next section, is whether such \gls{ldt} representations can be constructed in practice. We also stress that a polynomial number of arithmetic operations does not by itself imply polynomial-time complexity in the standard Random Access Memory model of computation \citep[Chap.~1]{aho1974design}. Indeed, \glspl{ldt} form a non-uniform computational model, in which the policy construction cost is not accounted for, and complexity is measured only by the number of linear tests performed at query time. Moreover, if the policy is represented explicitly as a tree, it must contain at least one leaf for each relevant optimality region, and therefore at least as many leaves as there are vertices of $\mathrm{conv}(\pset)$ that are optimal for some cost vector. For many combinatorial optimization problems, this number is exponential in $\pdim$. Nonetheless, our numerical experiments show that such \gls{ldt} policies can achieve substantially lower evaluation times than benchmark methods in useful regimes.

\subsection{Complexity of Policy Synthesis}
\label{sec:existence:hardness}

\Cref{prop:query-complexity} establishes the existence of \gls{ldt} policies with polynomial query complexity under a mild encoding assumption. This result, however, does not characterize the computational effort required to produce such a policy. We now investigate this question by characterizing the complexity of the following decision problem seeking a compact LDT policy representation.

\begin{problem}[\textsc{Ldt-s}]
    \label{prob:bounded-ldt}
    Consider an instance $(s,\pdim,\pset,\costdomain)$ where (i) $\pdim$ and $s$ are given in unary, and (ii) $\pset \subseteq \kZ^{\pdim}$ is finite and non-empty, and $\costdomain \subseteq \kQ^{\pdim}$ is non-empty, both with elements of polynomial encoding length in $\pdim$ and with polynomial-time membership operations ``$\pv \in \pset$'' and ``$\cv \in \costdomain$''. The problem \textsc{Ldt-s} asks whether there exists a \gls{ldt} $T$ with at most $s$ leaves such that
    \begin{equation}
        \label{eq:bounded-ldt-validity}
        \forall \cv \in \costdomain \; : \; T(\cv) \in \argmax_{\pv \in \pset} \ \transpose{\cv}\pv
    \end{equation}
    where $T(\cv) \in \pset$ denotes the solution returned by $T$ for some $\cv \in \costdomain$. All rational coefficients appearing in the \gls{ldt} branching tests and all solutions attached to its leaves are required to have polynomial encoding length.
\end{problem}
The restriction $\costdomain \subseteq \kQ^{\pdim}$ in \Cref{prob:bounded-ldt} is imposed to comply with standard encoding conventions and reflects finite-precision computation. Since a LDT with at most $s$ leaves contains at most $2s-1$ nodes, and since $s$ is given in unary, any candidate policy in \Cref{prob:bounded-ldt} has polynomial encoding length. The next result establishes the complexity of \textsc{Ldt-s}.

\begin{proposition}
    \label{prop:bounded-ldt}
    The decision problem \textsc{Ldt-s} is $\Sigma_2^p$-complete.
\end{proposition}
\begin{proof}
    We first establish membership in $\Sigma_2^p$, and then prove hardness by reduction from \textsc{Minimum DNF Tautology}  \citep[Problem~L7]{schaefer2002completeness}.
    \begin{itemize}
        \item \emph{Membership:} A $\Sigma_2^p$ procedure existentially guesses a \gls{ldt} $T$ with at most $s$ leaves. Since a LDT with at most $s$ leaves has at most $2s-1$ nodes, $s$ is given in unary, and all branching coefficients and leaf solutions have polynomial encoding length, this guess has polynomial length in the size of the instance. The procedure also verifies that every solution attached to a leaf of $T$ belongs to $\pset$. Now, the condition \eqref{eq:bounded-ldt-validity} rewrites as
        \begin{equation}
            \label{eq:bounded-ldt-quantifiers}
            \forall \cv \ \forall \pv :
            \left[
                \cv \in \costdomain \ \wedge \ \pv \in \pset
                \implies
                \transpose{\cv}T(\cv) \geq \transpose{\cv}\pv
            \right]
            ,
        \end{equation}
        where both quantifiers range over vectors of polynomial encoding length in $\pdim$ by the conventions of \Cref{prob:bounded-ldt}. For fixed $(T,\cv,\pv)$, the predicate between brackets can be decided in polynomial time: one checks membership in $\costdomain$ and in $\pset$, traverses $T$ to evaluate $T(\cv)$, and performs a single arithmetic comparison. The problem \textsc{Ldt-s} is therefore an existential guess of polynomial length followed by a $\forall\forall$ quantification of a polynomial-time predicate, which gives membership in $\Sigma_2^p$ \citep{stockmeyer1976polynomial}.

        \item \emph{Hardness:} \textsc{Minimum DNF Tautology} asks whether
        \begin{equation*}
            \varphi(\usat) = \bigvee_{k=1}^{K} D_k(\usat)
        \end{equation*}
        for a given tautological formula $\varphi$ in $3$-\textsc{dnf}
        over variables $\usat \in \{0,1\}^{p}$ and an integer $s$, there exists a subset of at most $s$ terms whose disjunction remains a tautology. We may assume, w.l.o.g., that the terms are distinct, consistent, and contain no repeated literals.
        For each term $D_k$, let $P_k$ and $N_k$ denote the indices of its positive and negative literals, respectively, and define $\ell_k = |P_k|+|N_k|$. Let $\mathbf{1}_{P_k},\mathbf{1}_{N_k} \in \{0,1\}^{p}$ denote the corresponding incidence vectors. We set $\pdim = 1+2p$ and define
        \begin{equation*}
            \begin{aligned}
                \pv^k
                &=
                \left(
                    1-2\ell_k,\,
                    2\mathbf{1}_{P_k},\,
                    2\mathbf{1}_{N_k}
                \right),
                && k \in \{1,\ldots,K\},
                \\
                \cv(\usat)
                &=
                \left(
                    1,\,
                    \usat,\,
                    \mathbf{1}_{p}-\usat
                \right),
                && \usat \in \{0,1\}^{p},
                \\
                \pset
                &=
                \kset{\pv^k}{k \in \{1,\ldots,K\}},
                &\quad
                \costdomain
                &=
                \kset{\cv(\usat)}{\usat \in \{0,1\}^{p}}.
            \end{aligned}
        \end{equation*}
        The resulting instance $(s,\pdim,\pset,\costdomain)$ of \textsc{Ldt-s} complies with the encoding conventions of \Cref{prob:bounded-ldt} and can be constructed in polynomial time. In particular, $\pset$ is explicitly listed, while membership in $\costdomain$ is checked by verifying that the last $p$ coordinates are the complements of the preceding $p$ binary coordinates.

        Denote by $m_k(\usat)$ the number of literals of $D_k$ that are \false{} under assignment $\usat$. We obtain
        \begin{equation}
            \label{eq:term-score}
            \transpose{\cv(\usat)}\pv^k
            =
            1-2\ell_k
            +2\sum_{i\in P_k}\usat_i
            +2\sum_{i\in N_k}(1-\usat_i)
            =
            1-2m_k(\usat).
        \end{equation}
        Consequently,
        \begin{equation}
            \label{eq:optimal-iff-term}
            \transpose{\cv(\usat)}\pv^k = 1
            \quad\Longleftrightarrow\quad
            D_k(\usat)=\true{},
        \end{equation}
        whereas $\transpose{\cv(\usat)}\pv^k \leq -1$ whenever $D_k(\usat)=\false{}$. Since $\varphi$ is a tautology, at least one term is \true{} for every assignment, and therefore
        \begin{equation}
            \label{eq:term-optimum}
            \max_{\pv\in\pset}
            \transpose{\cv(\usat)}\pv
            =
            1
            \qquad
            \forall \usat\in\{0,1\}^{p}.
        \end{equation}

        We are now in a position to establish the polynomial-time reduction between \textsc{Minimum DNF Tautology} and \textsc{Ldt-s}.
        \begin{itemize}[label=$\empty$]
            \item $(\impliedby)$ Suppose that \textsc{Minimum DNF Tautology} is \true{}, and let $J=\{j_1,\ldots,j_t\}$, with $t\leq s$, index terms whose disjunction is a tautology. Construct a \gls{ldt} that, for $h=1,\ldots,t-1$, leads to a leaf node associated with $\pv^{j_h}$ when $\transpose{\cv}\pv^{j_h}>0$, leads to a dummy leaf node when $\transpose{\cv}\pv^{j_h}=0$ since this condition can never be met for any $\cv \in \costdomain$ according to \eqref{eq:term-score}, and proceed to the next child node when $\transpose{\cv}\pv^{j_h}<0$. Its last leaf returns $\pv^{j_t}$. Moreover, at least one selected term is \true{} for every assignment, so the policy always returns a solution of value one. It is therefore valid by \eqref{eq:term-optimum} and has $t\leq s$ leaves. Hence, \textsc{Ldt-s} is \true{}.

            \item $(\implies)$ Suppose that \textsc{Ldt-s} is \true{}, and let $T$ be a valid \gls{ldt} with at most $s$ leaves. Let $J$ denote the indices of the distinct solutions attached to the reachable leaves of $T$. We have $|J|\leq s$. For any assignment $\usat$, validity of $T$ and \eqref{eq:term-optimum} imply that the solution returned by $T$ has objective value one. By \eqref{eq:optimal-iff-term}, its associated term is therefore \true{} under $\usat$. Hence,
            \begin{equation*}
                \bigvee_{k\in J} D_k(\usat)
                =
                \true{}
                \qquad
                \forall \usat\in\{0,1\}^{p}.
            \end{equation*}
            The terms indexed by $J$ thus form a tautology, and \textsc{Minimum DNF Tautology} is \true{}.
        \end{itemize}
    \end{itemize}
    Combining membership in $\Sigma_2^p$ and $\Sigma_2^p$-hardness yields that \textsc{Ldt-s} is $\Sigma_2^p$-complete.
\end{proof}

\subsection{Connection to Nearest Neighbor Search}
\label{sec:existence:nns}

We conclude this section by focusing on binary feasible sets $\pset \subseteq \{0,1\}^{\pdim}$. In this case, \Cref{prop:query-complexity} holds with 
\begin{equation}
    \treedepth \simeq 2\pdim^4\log_2(\pdim) + \bigO(\pdim^3)
\end{equation}
since $\treecond = 1$, ensuring that \gls{ldt} policies with polynomial query complexity exist without further assumptions. Moreover, the normal fan $\dualfaces(\pset)$ also admits a simple geometric interpretation in this setting, resulting from the following result.

\begin{proposition}
    \label{prop:nns}
    Define $\scalingfunc(\pv) = 2 \pv - \1$ and $\scalingfunc(\pset) = \kset{\scalingfunc(\pv)}{\pv \in \pset}$. For any binary feasible set $\pset \subseteq \{0,1\}^{\pdim}$, Problem~\eqref{prob:prob} can be equivalently formulated as
    \begin{equation}
        \label{prob:nns}
        \min_{\scalingfunc(\pv) \in \scalingfunc(\pset)} \norm{\scalingfunc(\pv) - \cv}{2}
        ,
    \end{equation}
    and $\opt{\pv} \in \pset$ is an optimal solution to Problem~\eqref{prob:prob} if and only if $\scalingfunc(\opt{\pv}) \in \scalingfunc(\pset)$ is an optimal solution to Problem~\eqref{prob:nns}.
\end{proposition}
\begin{proof}
    The objective function of Problem~\eqref{prob:nns} can be squared and expanded as $\norm{\scalingfunc(\pv) - \cv}{2}^2 = \norm{\scalingfunc(\pv)}{2}^2 - 2\transpose{\cv}\scalingfunc(\pv) + \norm{\cv}{2}^2$. Moreover, we have $\norm{\scalingfunc(\pv)}{2}^2 = \sum_{\idxentry=1}^{\pdim} (2\pvi{\idxentry} - 1)^2 = \sum_{\idxentry=1}^{\pdim} 4\pvi{\idxentry}^2 - 4\pvi{\idxentry} + 1^2 = \sum_{\idxentry=1}^{\pdim} 4\pvi{\idxentry} - 4\pvi{\idxentry} + 1 = \pdim$ since $\pv \in \{0,1\}^{\pdim}$. Hence, both terms $\norm{\scalingfunc(\pv)}{2}^2$ and $\norm{\cv}{2}^2$ are constant. Discarding them, eliminating the multiplicative factor $2$, and swapping the sign of the resulting objective function allows reformulating Problem~\eqref{prob:nns} as $\max_{\scalingfunc(\pv) \in \scalingfunc(\pset)} \ \transpose{\cv}\scalingfunc(\pv)$. Using the bijective variable shift $\pv \leftrightarrow \scalingfunc(\pv)$ finally leads to Problem~\eqref{prob:prob}.
\end{proof}

The above result shows that an \gls{ilp} with a binary feasible set $\pset \subseteq \{0,1\}^{\pdim}$ can always be cast as an exact \gls{nns} seeking the element in the transformed feasible set $\scalingfunc(\pset) \subseteq \{-1,1\}^{\pdim}$ closest to the queried cost vector $\cv \in \costdomain$. In this case, the cones of the normal fan $\dualfaces(\pset)$ defined in \eqref{eq:normal-fan-cone} can be expressed as
\begin{equation}
    F(\pv) = \kset{\cv \in \kR^{\pdim}}{\norm{\cv - \scalingfunc(\pv)}{{}} \leq \norm{\cv - \scalingfunc(\tilde{\pv})}{{}} \;\; \forall \tilde{\pv} \in \pset}
\end{equation}
for all $\pv \in \pset$, and $\dualfaces(\pset)$ coincides with the Voronoi diagram \citep{aurenhammer1991voronoi} induced by $\scalingfunc(\pset)$. Evaluating the policy proposed in \Cref{prop:df-policy} then amounts to performing an exact \gls{nns} query. This observation opens the door to policy designs that leverage the rich literature on \gls{nns} methods \citep{abbasifard2014survey}.

\section{Construction of Linear Decision Tree Policies}
\label{sec:dynamic}

Constructing a policy that matches the polynomial complexity bound of \Cref{prop:query-complexity} is challenging. Indeed, although of polynomial depth, the associated \gls{ldt} is expected to be large with at least as many leaves as there are vertices in $\convhull(\pset)$. \Cref{prop:bounded-ldt} then suggests a costly construction operation. Moreover, the results of Section~\ref{sec:existence} impose no restriction on the linear tests used for branching operations. In principle, one may therefore need to choose, at each node, among infinitely many candidate tests to be guaranteed to obtain the desired complexity bound.

\subsection{Practical Policy Structures}

In the sequel, we focus on \gls{ldt} policies whose branching operations are solely based on dividers from the set $\dividers(\pset)$. These dividers are finite in number and correspond to the actual boundaries between cost-vector regions associated with different optimal solutions in Problem~\eqref{prob:prob}. They therefore arise as natural candidates for branching tests within the \gls{ldt} structure. Moreover, we only consider \emph{binary} branching operations, although \Cref{prop:query-complexity} is stated for ternary ones. We assume a tie-breaking rule to handle cases where the queried cost vector satisfies $\divider_{\pv,\tilde{\pv}}(\cv) = 0$ for some $\divider_{\pv,\tilde{\pv}} \in \dividers(\pset)$, meaning that both feasible solutions $\pv \in \pset$ and $\tilde{\pv} \in \pset$ yield the same objective in Problem~\eqref{prob:prob}. Ternary structures are only instrumental in recovering \emph{all} optimal solutions using the policy defined in \Cref{prop:df-policy}, which is usually not required in practical applications. The \gls{ldt} structure of the policies considered can then be formally defined as follows.\footnote{For the sake of clarity, we drop the explicit reference to the feasible solutions in $\pset$ in the notation of cones in $\dualfaces(\pset)$ and dividers in $\dividers(\pset)$ when their definition is clear from the context.}
\begin{definition}
    \label{def:node-attributes}
    Given the normal fan $\dualfaces(\pset)$ associated with some $\pset \subseteq \kZ^{\pdim}$, we consider policies represented as \glspl{ldt} where each node corresponds to a pair of sets $\nodesymbol = \nodetuple$ containing all dividers $\divider \in \dividers(\pset)$ that have respectively yielded $\divider(\cv) < 0$ or $\divider(\cv) > 0$ up to this stage of the structure for the queried cost vector $\cv \in \costdomain$.
    The region, set of candidate cones, and dividers associated with node $\nodesymbol$ are defined as
    \begin{align}
        \label{eq:node-region}
        \noderegion(\nodesymbol) &= \kset{\cv \in \costdomain}{\divider(\cv) < 0 \ \forall \divider \in \nodelt \;\;\text{and}\;\; \divider(\cv) > 0 \ \forall \divider \in \nodegt}, \\
        \label{eq:node-cones}
        \dualfaces(\nodesymbol) &= \kset{\dualface \in \dualfaces(\pset)}{\noderegion(\nodesymbol) \cap \dualface \neq \emptyset}, \\
        \label{eq:node-dividers}
        \dividers(\nodesymbol) &= \kset{\divider \in \dividers(\pset)}{\noderegion(\nodesymbol) \cap \divider \neq \emptyset},
    \end{align}
    respectively, and any $\divider \in \dividers(\pset)$ induces two children $\childlt{\nodesymbol}{\divider}$ and $\childgt{\nodesymbol}{\divider}$, obtained by appending this element to $\nodelt$ and $\nodegt$, respectively.
\end{definition}

Starting from the root $\noderoot=(\emptyset,\emptyset)$ representing the entire cost domain region $\noderegion(\noderoot) = \costdomain$, each branching test determines the sign of the selected divider when evaluated at the queried $\cv \in \costdomain$, thereby discarding cones of the normal fan that cannot contain this cost vector. The set $\dualfaces(\nodesymbol)$ indicates the cones that may still contain the queried cost vector at node $\nodesymbol$, while $\dividers(\nodesymbol)$ corresponds to dividers that can be used to further refine its localization. As deeper nodes are explored, the cardinality of these sets naturally decreases. By design, each branching operation discards at least one cone and one divider. Eventually, a leaf node associated with a unique cone of the normal fan $\dualfaces(\pset)$ remains, which precisely characterizes an optimal solution to Problem~\eqref{prob:prob} corresponding to the input cost vector $\cv \in \costdomain$ through \Cref{prop:df-policy}. \Cref{fig:ldt} illustrates an example of such an \gls{ldt} policy structure.

\begin{figure}[htbp]
    \centering
    \tikzset{
    level 1/.style={level distance=
    1.5cm, sibling distance=3.5cm},
    level 2/.style={level distance=1.5cm, sibling distance=1.75cm},
    treenode/.style={draw, rectangle, rounded corners, thick, minimum size=4mm, inner sep=2pt, font=\small},
    leafnode/.style={draw, circle, thick, minimum size=4mm, inner sep=2pt, font=\small},
    edge from parent/.style={-{Stealth}, draw=black},
    edge label/.style={midway, above, sloped, font=\small, yshift=-2pt},
}

\begin{tikzpicture}
    \begin{scope}[xshift=-0.25\linewidth,scale=1.25]
        \coordinate (O) at (0,0);
        \coordinate (P1) at (-1,-1);
        \coordinate (P2) at (-1,0);
        \coordinate (P3) at (0,1);
        \coordinate (P4) at (1,0);
        \coordinate (P5) at (2,0);
        \coordinate (H12) at (-2.2,0);
        \coordinate (H23) at (-2.2,2.2);
        \coordinate (H35) at (1.15,2.2);
        \coordinate (H15) at (0.8,-2.2);

        \filldraw[BurntOrange, opacity=0.25] (0,0) -- (H12) -- (-2.2,-2.2) -- (H15) -- cycle;
        \filldraw[Cerulean, opacity=0.25] (0,0) -- (H12) -- (H23) -- cycle;
        \filldraw[Red, opacity=0.25] (0,0) -- (H23) -- (H35) -- cycle;
        \filldraw[ForestGreen, opacity=0.25] (0,0) -- (H15) -- (2.2,-2.2) -- (2.2,2.2 ) -- (H35) -- cycle;

        \draw[-{Stealth}] (0,-2.25) -- (0,2.25);
        \draw[-{Stealth}] (-2.25,0) -- (2.25,0);
        \foreach \x in {-2,-1,1,2} {
            \draw[gray, dotted] (\x,-2.25) -- (\x,2.25);
            \draw[thick] (\x,-0.075) -- (\x,0.075);
        }
        \foreach \y in {-2,-1,1,2} {
            \draw[gray, dotted] (-2.25,\y) -- (2.25,\y);
            \draw[thick] (-0.075,\y) -- (0.075,\y);
        }

        \draw[gray] (-1,0) rectangle ($(-1,0)+(0.125,-0.125)$);
        \draw[gray, rotate around={-45:(-0.5,0.5)}] (-0.5,0.5) rectangle ($(-0.5,0.5)+(0.125,-0.125)$); 
        \draw[gray, rotate around={-117:(0.4125,0.7925)}] (0.4125,0.7925) rectangle ($(0.4125,0.7925)+(0.125,-0.125)$);
        \draw[gray, rotate around={-70:(0.22,-0.59)}] (0.22,-0.59) rectangle ($(0.22,-0.59)+(-0.125,0.125)$); 

        \draw[opacity=0] (O) -- (H12) node[gray,opacity=1,midway,above,sloped,font=\scriptsize,xshift=-0.75cm,yshift=0pt] {$\transpose{(\pv_1 - \pv_2)}\cv=0$};
        \draw[opacity=0] (O) -- (H23) node[gray,opacity=1,midway,above,sloped,font=\scriptsize,xshift=-0.85cm,yshift=-2pt] {$\transpose{(\pv_2 - \pv_3)}\cv=0$};
        \draw[opacity=0] (O) -- (H35) node[gray,opacity=1,midway,above,sloped,font=\scriptsize,xshift=0.65cm,yshift=-2pt] {$\transpose{(\pv_3 - \pv_5)}\cv=0$};
        \draw[opacity=0] (O) -- (H15) node[gray,opacity=1,midway,above,sloped,font=\scriptsize,xshift=0.65cm,yshift=-2pt] {$\transpose{(\pv_1 - \pv_5)}\cv=0$};
        
        \draw[densely dashed] (P1) -- (P2) -- (P3) -- (P5) -- cycle;
        \node at (1.25,0.75) {$\convhull(\pset)$};
        
        \fill (P1) circle (1.25pt) node[below] {$\pv_1$};
        \fill (P2) circle (1.25pt) node[below left] {$\pv_2$};
        \fill (P3) circle (1.25pt) node[left] {$\pv_3$};
        \fill (P4) circle (1.25pt) node[above] {$\pv_4$};
        \fill (P5) circle (1.25pt) node[below] {$\pv_5$};

        \node[BurntOrange, above right] at (-2,-2) {$\dualface(\pv_1)$};
        \node[Cerulean, right] at (-2,0.75) {$\dualface(\pv_2)$};
        \node[Red, below] at (-0.5,2) {$\dualface(\pv_3)$};
        \node[ForestGreen, above left] at (2,-2) {$\dualface(\pv_5)$};

    \end{scope}
    \begin{scope}[xshift=0.24\linewidth]
        \node[treenode] (N0) at (0,2cm) {$\transpose{(\pv_1 - \pv_2)}\cv$}
            child {node[treenode] (N1) {$\transpose{(\pv_1 - \pv_5)}\cv$}
                child {node[leafnode] (N3) {$\pv_1$}
                edge from parent node[edge label] {$>0$}}
                child {node[leafnode] (N4) {$\pv_5$}
                edge from parent node[edge label] {$<0$}}
            edge from parent node[edge label] {$>0$}}
            child {node[treenode] (N2) {$\transpose{(\pv_3 - \pv_5)}\cv$}
                child {node[treenode] (N5) {$\transpose{(\pv_2 - \pv_3)}\cv$}
                    child {node[leafnode] (N7) {$\pv_2$}
                    edge from parent node[edge label] {$>0$}}
                    child {node[leafnode] (N8) {$\pv_3$}
                    edge from parent node[edge label] {$<0$}}
                edge from parent node[edge label] {$>0$}}
                child {node[leafnode] (N6) {$\pv_5$}
                edge from parent node[edge label] {$<0$}}
            edge from parent node[edge label] {$<0$}}
            ;
    
        \node[gray, font=\scriptsize, anchor=south, xshift=12pt, yshift=-2pt,xshift=-10pt] at (N0.north east) {$\{\pv_{1},\pv_{2},\pv_{3},\pv_{5}\}$};
        \node[gray, font=\scriptsize, anchor=south, yshift=-2pt, xshift=12pt] at (N1.north west) {$\{\pv_{1},\pv_{5}\}$};
        \node[gray, font=\scriptsize, anchor=south, yshift=-2pt, xshift=-15pt] at (N2.north east) {$\{\pv_{2},\pv_{3},\pv_{5}\}$};
        \node[gray, font=\scriptsize, anchor=south, yshift=-2pt, xshift=12pt] at (N5.north west) {$\{\pv_2,\pv_3\}$};
    
        \node[font=\small] (c) at ($(N0.north)+(0,0.75)$) {Query $\cv \in \kR^2$};
        \draw[-{Stealth}] (c) -- ($(N0.north)+(0,0.05)$);
    \end{scope}
\end{tikzpicture}
    \caption{Left: Two-dimensional feasible set $\pset = \{\pv_1,\dots,\pv_5\}$ and its associated normal fan. Its cones are separated by hyperplanes normal to the faces of $\convhull(\pset)$ and passing through the origin. Each cone $\dualface(\pv_i)$ corresponds to the set of cost vectors for which $\pv_i \in \pset$ is an optimal solution to Problem~\eqref{prob:prob}. Note that $\pv_4 \in \mathrm{int}\convhull(\pset)$, so this feasible solution is never optimal for any cost vector, and its associated cone is empty. Right: \gls{ldt} structure used to locate a queried cost vector within $\costdomain = \kR^{\pdim}$ in the normal fan $\dualfaces(\pset) = \{\dualface(\pv_i)\}_{\pv_i \in \pset}$. Gray annotations indicate candidate solutions at each internal node, and the leaf reached yields an optimal solution to Problem~\eqref{prob:prob}.}
    \label{fig:ldt}
\end{figure}

\subsection{Dynamic-Programming-and-Pruning Algorithm}
\label{sec:dynamic:dynprog}

Among the family of \glspl{ldt} specified by \Cref{def:node-attributes}, those of minimal depth are of particular interest. They ensure, in the worst case, a minimal number of arithmetic operations during their traversal, thereby achieving the best possible query complexity for the associated policy. To construct such a structure, denote by $\nodeheight(\nodesymbol)$ the \emph{minimal height} of a node $\nodesymbol$, that is, the minimal number of branching operations to reach the deepest leaf it can attain. This quantity can be expressed recursively as
\begin{equation}
    \label{eq:minimal-node-depth}
    \nodeheight(\nodesymbol) =
    \begin{cases}
        0 & \text{if } \card{\dualfaces(\nodesymbol)} = 1, \\
        \min_{\divider \in \dividers(\nodesymbol)}
        \max \bigl\{
            \nodeheight(\childlt{\nodesymbol}{\divider}),
            \nodeheight(\childgt{\nodesymbol}{\divider})
        \bigr\} + 1
        & \text{otherwise},
    \end{cases}
\end{equation}
by selecting the divider among the remaining ones yielding the minimal height among its associated children, with an additional unit accounting for the branching operation performed.
If $\card{\dualfaces(\nodesymbol)} = 1$, then node $\nodesymbol$ is a leaf and its minimal height is $\nodeheight(\nodesymbol)=0$.
Since the root node $\noderoot = (\emptyset,\emptyset)$ is common to all \glspl{ldt} characterized by \Cref{def:node-attributes}, evaluating $\nodeheight(\noderoot)$ directly yields the minimal depth achievable among all candidate structures, as well as the branching test that should be selected at each node to attain it. Owing to its recursive definition, the function $\nodeheight$ can be evaluated using a standard dynamic programming procedure.\\

\paragraph{Pruning test}
Because $\card{\dividers(\pset)} < +\infty$, the number of possible nodes to process during the dynamic programming procedure is finite, and the evaluation of $\nodeheight(\noderoot)$ terminates in a finite number of steps. Nevertheless, processing each node $\nodesymbol$ in the recursion requires testing all remaining dividers $\divider \in \dividers(\nodesymbol)$ for the branching operation. This can lead to substantial computational effort, particularly at shallower nodes where $\card{\dividers(\nodesymbol)}$ is large. To reduce the effort of constructing a minimal-depth \gls{ldt}, a pruning mechanism is used to discard nodes that cannot belong to a minimal-depth structure. Specifically, denote by $\nodedepth(\nodesymbol) = \card{\nodelt} + \card{\nodegt}$ the \emph{depth} of a node $\nodesymbol = \nodetuple$, that is, the number of branching operations performed from the root. Then, a node $\nodesymbol$ verifying the pruning condition
\begin{equation}
    \nodedepth(\nodesymbol) + \nodeheight(\nodesymbol) > \nodeheight(\noderoot)
\end{equation}
cannot belong to a minimal-depth \gls{ldt} since the depth of any tree passing through it would exceed the minimal depth achievable among all candidate structures. Any node verifying this pruning condition can then be discarded during the dynamic programming procedure.\\

\paragraph{Pruning bounds}
Although the quantities $\nodeheight(\nodesymbol)$ and $\nodeheight(\noderoot)$ are unknown when processing node $\nodesymbol$, a surrogate pruning condition
\begin{equation}
    \label{eq:pruning}
    \nodedepth(\nodesymbol) + \sampled{\nodeheight}{\text{lb}}(\nodesymbol)
    >
    \nodeheight^{\text{ub}}(\noderoot)
\end{equation}
can still be implemented in practice using a lower bound $\sampled{\nodeheight}{\text{lb}}(\nodesymbol) \leq \nodeheight(\nodesymbol)$ and an upper bound $\nodeheight^{\text{ub}}(\noderoot) \geq \nodeheight(\noderoot)$. A valid lower bound on the minimal height $\nodeheight(\nodesymbol)$ at any node is given by
\begin{equation}
    \label{eq:lower-bound}
    \sampled{\nodeheight}{\text{lb}}(\nodesymbol)
    =
    \ceil{\log_2\bigl(\card{\dualfaces(\nodesymbol)}\bigr)}
\end{equation}
since in the ideal case, each branching operation perfectly separates candidate cones in $\dualfaces(\nodesymbol)$ into two subsets of equal cardinality, until a leaf is reached. Moreover, an upper bound on $\nodeheight(\noderoot)$ can be constructed from the quantity
\begin{equation}
    \label{eq:minimal-node-depth-sampled}
    \sampled{\nodeheight}{\hyperplanenumber}(\nodesymbol)
    =
    \begin{cases}
        0
        & \text{if } \card{\dualfaces(\nodesymbol)} = 1, \\
        \min_{\divider \in \sampled{\dividers}{\hyperplanenumber}(\nodesymbol)}
        \max \bigl\{
            \sampled{\nodeheight}{\hyperplanenumber}(\childlt{\nodesymbol}{\divider}),
            \sampled{\nodeheight}{\hyperplanenumber}(\childgt{\nodesymbol}{\divider})
        \bigr\} + 1
        & \text{otherwise},
    \end{cases}
\end{equation}
where $\sampled{\dividers}{\hyperplanenumber}(\nodesymbol)$ denotes the restriction of the set $\dividers(\nodesymbol)$ to its first $\hyperplanenumber \in \kN$ elements, using the convention that $\sampled{\dividers}{\hyperplanenumber}(\nodesymbol) = \dividers(\nodesymbol)$ when $\hyperplanenumber \geq \card{\dividers(\nodesymbol)}$.
In particular, we have
\begin{equation}
    \label{eq:upper-bound-monotonicity}
    \hyperplanenumber \leq \hyperplanenumber'
    \implies
    \sampled{\dividers}{\hyperplanenumber}(\nodesymbol) \subseteq \sampled{\dividers}{\hyperplanenumber'}(\nodesymbol)
    \implies
    \sampled{\nodeheight}{\hyperplanenumber}(\nodesymbol) \geq \sampled{\nodeheight}{\hyperplanenumber'}(\nodesymbol)
\end{equation}
for all $\nodesymbol$.
Setting $\hyperplanenumber=1$ corresponds to a greedy processing of the node where a single divider is considered for the branching operation, whereas with $\hyperplanenumber=\card{\dividers(\pset)}$ one recovers $\sampled{\nodeheight}{\hyperplanenumber}(\nodesymbol)=\nodeheight(\nodesymbol)$ for all nodes $\nodesymbol$.\\

\paragraph{Iterative construction}
Based on these observations, we propose an iterative construction strategy in which $\sampled{\nodeheight}{\hyperplanenumber}(\noderoot)$ is evaluated recursively from \eqref{eq:minimal-node-depth-sampled} for a sequence of increasing $\hyperplanenumber \in \{\hyperplanenumber_1,\hyperplanenumber_2,\dots,\hyperplanenumber_{\max}\}$.
This produces \glspl{ldt} of non-increasing depth
\begin{equation}
    \label{eq:decreasing-depth}
    \sampled{\nodeheight}{\hyperplanenumber_{1}}(\noderoot)
    \geq
    \sampled{\nodeheight}{\hyperplanenumber_{2}}(\noderoot)
    \geq
    \dots
    \geq
    \sampled{\nodeheight}{\hyperplanenumber_{\max}}(\noderoot)
    ,
\end{equation}
and setting $\hyperplanenumber_{\max} = \card{\dividers(\pset)}$ guarantees that the last one achieves minimal depth. As the iterations progress, more candidate dividers are considered for branching operations in each node processed, expanding the search space. The pruning condition \eqref{eq:pruning} can be applied using the lower bound \eqref{eq:lower-bound} together with the upper bound recovered from \eqref{eq:decreasing-depth}, which is progressively refined over the iterations. Evaluating $\sampled{\nodeheight}{\hyperplanenumber}(\noderoot)$ for small values of $\hyperplanenumber$ is typically inexpensive since only a limited number of dividers are examined for branching operations. For larger values of $\hyperplanenumber$, the evaluation is more costly, but pruning becomes more effective and can potentially offset part of the computational load. Importantly, once a node is pruned for a given value of $\hyperplanenumber$, it remains pruned for all values that are considered subsequently due to Property~\eqref{eq:upper-bound-monotonicity}. Moreover, the sets $\dualfaces(\nodesymbol)$ and $\dividers(\nodesymbol)$ are independent of $\hyperplanenumber$, and therefore need to be constructed only the first time a node $\nodesymbol$ is processed. \Cref{alg:buildtree-sampling} summarizes this iterative construction strategy. Any pruned node is assigned $\nodeheight(\nodesymbol)=+\infty$ to ensure that no divider leading to this node will be selected in subsequent steps.

\begin{algorithm}[htbp]
    \caption{Minimal-depth LDT synthesis}
    \label{alg:buildtree-sampling}
    \algrenewcommand\algorithmicrequire{\textbf{Input:}}
\algrenewcommand\algorithmicensure{\textbf{Output:}}

\begin{algorithmic}[1]

\Require Increasing sequence $\hyperplanenumber \in \{\hyperplanenumber_1,\dots,\hyperplanenumber_{\max}\}$ with $\hyperplanenumber_{\max} = \card{\dividers(\pset)}$
\Ensure Minimal depth achievable $\nodeheight(\noderoot)$ among \gls{ldt} structures

\vspace{0.5em}

\State Initialize $\noderoot = (\emptyset, \emptyset)$ and $\sampled{\nodeheight}{\mathrm{ub}}(\noderoot) = +\infty$

\For{$\hyperplanenumber \in \{\hyperplanenumber_1,\dots,\hyperplanenumber_{\max}\}$}
    \State $\sampled{\nodeheight}{\mathrm{ub}}(\noderoot) \gets \textsc{Evaluate}(\noderoot, \hyperplanenumber, \sampled{\nodeheight}{\mathrm{ub}}(\noderoot))$
\EndFor

\State \Return $\sampled{\nodeheight}{\mathrm{ub}}(\noderoot)$

\vspace{0.5em}

\Function{Evaluate}{$\nodesymbol, \hyperplanenumber, \sampled{\nodeheight}{\mathrm{ub}}(\noderoot)$}

    \If{$\nodeheight(\nodesymbol)$ has been memoized}
        \Comment{Node memoization}
        \State \Return $\nodeheight(\nodesymbol)$
    \EndIf

    \If{$\nodesymbol$ has not been evaluated yet}
        \Comment{Node initialization}
        \State Compute $\dualfaces(\nodesymbol)$ and $\dividers(\nodesymbol)$
    \EndIf

    \vspace{0.5em}

    \If{$\card{\dualfaces(\nodesymbol)} = 1$}
        \Comment{Leaf test}
        \State $\nodeheight(\nodesymbol) \gets 0$
        \State Memoize $\nodeheight(\nodesymbol)$
        \State \Return $\nodeheight(\nodesymbol)$
    \EndIf

    \If{$\nodedepth(\nodesymbol) + \nodeheight^{\mathrm{lb}}(\nodesymbol) > \sampled{\nodeheight}{\mathrm{ub}}(\noderoot)$}
        \Comment{Pruning test}
        \State $\nodeheight(\nodesymbol) \gets +\infty$
        \State Memoize $\nodeheight(\nodesymbol)$
        \State \Return $\nodeheight(\nodesymbol)$
    \EndIf

    \vspace{0.5em}

    \State $\sampled{\nodeheight}{\hyperplanenumber}(\nodesymbol) \gets +\infty$
    \Comment{DP recursion}
    \For{$\divider{}{} \in \sampled{\dividers}{\hyperplanenumber}(\nodesymbol)$}
        \State $\sampled{\nodeheight}{\hyperplanenumber}(\childlt{\nodesymbol}{\divider}) \gets \textsc{Evaluate}(\childlt{\nodesymbol}{\divider}, \hyperplanenumber, \sampled{\nodeheight}{\mathrm{ub}}(\noderoot))$
        \State $\sampled{\nodeheight}{\hyperplanenumber}(\childgt{\nodesymbol}{\divider}) \gets \textsc{Evaluate}(\childgt{\nodesymbol}{\divider}, \hyperplanenumber, \sampled{\nodeheight}{\mathrm{ub}}(\noderoot))$

        \State $\sampled{\nodeheight}{\hyperplanenumber}(\nodesymbol) \gets 
        \min\big\{
            \sampled{\nodeheight}{\hyperplanenumber}(\nodesymbol),\ 
            \max\big(
                \sampled{\nodeheight}{\hyperplanenumber}(\childlt{\nodesymbol}{\divider}),
                \sampled{\nodeheight}{\hyperplanenumber}(\childgt{\nodesymbol}{\divider})
            \big) + 1
        \big\}$
    \EndFor

    \vspace{0.5em}

    \If{$\hyperplanenumber \geq \card{\dividers(\nodesymbol)}$}
        \Comment{Node memoization}
        \State $\nodeheight(\nodesymbol) \gets \sampled{\nodeheight}{\hyperplanenumber}(\nodesymbol)$
        \State Memoize $\nodeheight(\nodesymbol)$
    \EndIf

    \State \Return $\sampled{\nodeheight}{\hyperplanenumber}(\nodesymbol)$

\EndFunction

\end{algorithmic}
\end{algorithm}

\subsection{Acceleration Strategies}
\label{sec:dynamic:acceleration}

Beyond pure arithmetic operations, the main computational effort when evaluating a node $\nodesymbol$ in \Cref{alg:buildtree-sampling} lies in recovering its associated sets $\dualfaces(\nodesymbol)$ and $\dividers(\nodesymbol)$. Computing them from their characterization in \Cref{def:node-attributes} requires checking the feasibility of $\card{\dualfaces} + \card{\dividers}$ linear systems of inequalities. However, since $\noderegion(\nodesymbol) \subseteq \noderegion(\nodesymbol')$ when $\nodesymbol$ is a child of node $\nodesymbol'$ in the \gls{ldt} structure, this effort can be reduced to checking the feasibility of $\card{\dualfaces(\nodesymbol')} + \card{\dividers(\nodesymbol')}$ linear systems by propagating these sets from parent to child nodes.
Using this mechanism, the number of linear system feasibility checks required to recover the set of cones and dividers naturally decreases as deeper nodes in the \gls{ldt} are explored. Nevertheless, the computational load dedicated to these operations can remain substantial. We therefore propose additional acceleration strategies to further reduce the overall computational effort for synthesizing a minimal-depth \gls{ldt} via \Cref{alg:buildtree-sampling}.\\

\paragraph{Inference rules for cones}
When evaluating a divider $\divider \in \dividers(\nodesymbol)$ at a node~$\nodesymbol$, \Cref{def:node-attributes} yields 
\begin{equation}
    \label{eq:intersection-decomposition}
    \noderegion(\nodesymbol') \cap \dualface =
    \left\{
    \begin{array}{ll}
        \noderegion(\nodesymbol) \cap \dualface{} \cap \divider_{<} 
        & \text{for child } \nodesymbol' = \childlt{\nodesymbol}{\divider}, \\
        \noderegion(\nodesymbol) \cap \dualface{} \cap \divider_{>} 
        & \text{for child } \nodesymbol' = \childgt{\nodesymbol}{\divider},
    \end{array}
    \right.
\end{equation}
for any candidate cone $\dualface \in \dualfaces(\nodesymbol)$, where we define $\divider_{\odot} = \kset{\cv \in \kR^{\pdim}}{\divider(\cv) \odot 0}$ for all $\odot \in \{<,\leq,=,\geq,>\}$. Interestingly, the emptiness of the intersections $\dualface \cap \divider_{<}$ and $\dualface \cap \divider_{>}$ in \eqref{eq:intersection-decomposition} does not depend on the node considered. Therefore, by preprocessing the relative position of all cones in $\dualfaces(\pset)$ with respect to all dividers in $\dividers(\pset)$, one can infer whether a cone $\dualface \in \dualfaces(\nodesymbol)$ must remain a candidate in the children $\childlt{\nodesymbol}{\divider}$ and $\childgt{\nodesymbol}{\divider}$ of node $\nodesymbol$ or can be discarded.
More precisely, the following result holds.
\begin{proposition}
    \label{prop:filtering-rules}
    We have
    \begin{subequations}
        \begin{align}
            \label{eq:filtering-rules-leq}
            \dualface{} \subseteq \divider_{\leq} &\implies \noderegion(\childlt{\nodesymbol}{\divider}) \cap \dualface{} \neq \emptyset \quad \text{and} \quad \noderegion(\childgt{\nodesymbol}{\divider}) \cap \dualface{} = \emptyset \\
            \label{eq:filtering-rules-geq}
            \dualface{} \subseteq \divider_{\geq} &\implies \noderegion(\childlt{\nodesymbol}{\divider}) \cap \dualface{} = \emptyset \quad \text{and} \quad \noderegion(\childgt{\nodesymbol}{\divider}) \cap \dualface{} \neq \emptyset
        \end{align}
    \end{subequations}
    for any node $\nodesymbol \neq \nodesymbol_0$, candidate cone $\dualface \in \dualfaces(\nodesymbol)$, and divider $\divider \in \dividers(\nodesymbol)$. If $\costdomain$ is open, this result also holds for the root node $\nodesymbol = \nodesymbol_0$.
\end{proposition}
\begin{proof}
    Consider the child node $\childgt{\nodesymbol}{\divider}$. On the one hand, if $\dualface \subseteq \divider_{\leq}$, then $\dualface \cap \divider_{>} = \emptyset$. It follows from \eqref{eq:intersection-decomposition} that $\noderegion(\childgt{\nodesymbol}{\divider}) \cap \dualface{} = \emptyset$. On the other hand, if $\dualface{} \subseteq \divider_{\geq}$, then
    \begin{subequations}
        \begin{align}
            \noderegion(\childgt{\nodesymbol}{\divider}) \cap \dualface{} &= \noderegion(\nodesymbol) \cap \dualface{} \cap \divider_{>} \\
            &= \noderegion(\nodesymbol) \cap \dualface{} \cap \divider_{\geq} \setminus \divider_{=} \\
            &= \noderegion(\nodesymbol) \cap \dualface{} \setminus \divider_{=}
            .
        \end{align}
    \end{subequations}
    Moreover, we have $\noderegion(\nodesymbol) \cap \dualface{} \neq \emptyset$ since $\dualface \in \dualfaces(\nodesymbol)$, the region $\noderegion(\nodesymbol)$ is open by \Cref{def:node-attributes} since either $\nodesymbol \neq \nodesymbol_0$ or $\nodesymbol = \nodesymbol_0$ and that $\costdomain$ is open, and $\dualface{}$ is full-dimensional and closed by definition \eqref{eq:normal-fan-cone}. Hence, the intersection $\noderegion(\nodesymbol) \cap \dualface{}$ is a non-empty open subset of $\mathbb{R}^{\pdim}$. Since $\divider_{=}$ is a linear subspace, it cannot cover $\noderegion(\nodesymbol) \cap \dualface{}$ entirely. We conclude that $\noderegion(\nodesymbol) \cap \dualface{} \setminus \divider_{=} \neq \emptyset$, which yields $\noderegion(\childgt{\nodesymbol}{\divider}) \cap \dualface{} \neq \emptyset$.
    The reasoning for the other child node $\childlt{\nodesymbol}{\divider}$ holds symmetrically.
\end{proof}
Using the inference rules provided by \Cref{prop:filtering-rules}, some of the candidate cones associated with the children of the node $\nodesymbol$ can be determined at no cost, without explicitly testing the emptiness of the intersections $\noderegion(\childlt{\nodesymbol}{\divider}) \cap \dualface$ and $\noderegion(\childgt{\nodesymbol}{\divider}) \cap \dualface$ via linear system feasibility checks. In particular, only the cones in $\dualfaces(\nodesymbol)$ that straddle the divider $\divider \in \dividers(\nodesymbol)$ selected for branching need to be checked explicitly to determine whether they remain candidates in the child nodes.\\

\paragraph{Inference rules for dividers}
Similarly, determining whether some $\divider \in \dividers(\nodesymbol)$ remains a divider in the children $\childlt{\nodesymbol}{\divider}$ and $\childgt{\nodesymbol}{\divider}$ of node $\nodesymbol$ further reduces the number of feasibility checks performed in \Cref{alg:buildtree-sampling}. To derive such inference rules when branching on some divider $\divider \in \dividers(\pset)$, we observe that if
\begin{equation}
    \label{eq:inference-nonempty}
    \dualfaces(\childlt{\nodesymbol}{\divider}) \neq \emptyset
    \quad\text{and}\quad
    \dualfaces(\childgt{\nodesymbol}{\divider}) \neq \emptyset,
\end{equation}
then at least one cone remains in both children of node $\nodesymbol$, implying that $\noderegion(\nodesymbol) \cap \divider \neq \emptyset$. It then follows from \Cref{def:node-attributes} that $\divider \in \dividers(\nodesymbol)$.
Conversely, if
\begin{equation}
    \label{eq:inference-empty}
    \dualfaces(\childlt{\nodesymbol}{\divider})
    =
    \emptyset
    \quad\text{or}\quad
    \dualfaces(\childgt{\nodesymbol}{\divider})
    =
    \emptyset,
\end{equation}
then all cones in $\dualfaces(\nodesymbol)$ lie on the same side of this divider, which gives $\noderegion(\nodesymbol) \cap \divider = \emptyset$ and implies that $\divider \notin \dividers(\nodesymbol)$ from \Cref{def:node-attributes}. Using these rules, part of the set $\dividers(\nodesymbol)$ can be quickly determined without checking the emptiness of $\noderegion(\nodesymbol) \cap \divider$ through a linear feasibility problem. This operation requires the knowledge of candidate cones $\dualfaces(\childlt{\nodesymbol}{\divider})$ and $\dualfaces(\childgt{\nodesymbol}{\divider})$ associated with the children of node $\nodesymbol$. Since these will be needed later anyway in \Cref{alg:buildtree-sampling} whenever $\divider \in \dividers(\nodesymbol)$, we compute and cache them in advance, leveraging the cone inference rules. The cones associated with these children will then be directly available when processing them for the first time.\\

\paragraph{Sorting node dividers}
When evaluating a node $\nodesymbol$, the candidate dividers in $\dividers(\nodesymbol)$ can be ordered to prioritize the most promising branching decisions during the iterative construction of \Cref{alg:buildtree-sampling}.
In particular, dividers that induce a more balanced partition of the cones in $\dualfaces(\nodesymbol)$ across the child nodes can be favored, as they are more likely to lead quickly to leaf nodes.
This ordering can be based on a function $\kfuncdef{\discrepfunc{\nodesymbol}}{\dividers(\nodesymbol)}{\kR}$ assigning a \emph{discrepancy score} that measures how balanced the partition of cones in $\dualfaces(\nodesymbol)$ induced by a divider is.
Several discrepancy measures from the literature may be used \citep[Table~1]{jost2006entropy}.
We propose to rely on the variance-based measure
\begin{equation}
    \label{eq:discrepancy-variance}
    \discrepfunc{\nodesymbol}(\divider)
    =
    \left(
        \card{\dualfaces(\childlt{\nodesymbol}{\divider})}
        -
        \frac{1}{2}\card{\dualfaces(\nodesymbol)}
    \right)^2
    +
    \left(
        \card{\dualfaces(\childgt{\nodesymbol}{\divider})}
        -
        \frac{1}{2}\card{\dualfaces(\nodesymbol)}
    \right)^2
    ,
\end{equation}
which captures the deviation from a perfectly balanced partition of the cones between the two child nodes.\footnote{Note that the sets $\dualfaces(\childlt{\nodesymbol}{\divider})$ and $\dualfaces(\childgt{\nodesymbol}{\divider})$ are not disjoint in general.}
The candidate cones $\dualfaces(\childlt{\nodesymbol}{\divider})$ and $\dualfaces(\childgt{\nodesymbol}{\divider})$ of the children required to compute this score are readily available when performing this sorting operation at node $\nodesymbol$, as they are already constructed and cached to evaluate divider inference rules.

\section{Experimental Analyses}
\label{sec:numerics}

In the previous sections, we have presented a generic family of \gls{ldt} policies for \glspl{ilp} along with a practical procedure to construct them. We analyze the insights into the inherent complexity of \glspl{ilp} provided by their associated \gls{ldt} policies, assess the efficiency of evaluating these policies relative to standard optimization methods, and examine their construction cost.

\subsection{Experimental Setup}
\label{sec:numerics:instances}

We conduct our experiments on several canonical classes of \glspl{ilp}, each characterized by a fixed feasible-set structure and by a parameter $\sdim \in \kN$ driving the problem dimension. For instances involving minimization instead of maximization, we negate the cost vector to match the framework of Problem~\eqref{prob:prob}.
The selected classes are as follows:
\begin{itemize}
    \item \textsc{Knp}: Knapsack problems seeking a packing of $\sdim \in \kN$ items with weights $\knpweight \in \kR_+^{\sdim}$ of maximum value by solving Problem~\eqref{prob:prob} with
    \begin{equation*}
        \pset = \left\{\pv \in \{0,1\}^{\pdim} \mid \textstyle \transpose{\knpweight}\pv \leq \knpcapacity \right\}
    \end{equation*}
    where $\pdim = \sdim$, the value $\knpcapacity > 0$ denotes a maximum weight capacity, and the cost vector in the domain $\costdomain = \kR_+^{\sdim}$ encodes the value of each item.
    We consider specific instances where $\knpweight = \{1,\dots,\sdim\}$ and $\knpcapacity = \sdim$ to obtain a fixed feasible set for each value of parameter $\sdim \in \kN$, for which the number of feasible solutions scales as $\card{\pset} \simeq \bigO\big(\sdim^{-3/4}\exp(\pi\sqrt{\sdim/3})\big)$.\footnote{See \url{https://oeis.org/A026906}.}\\
    \item \textsc{Tsp}: Traveling salesman problems seeking a Hamiltonian cycle of minimum length over vertices $\mathcal{G}_V = \{1,\dots,\sdim\}$ in an undirected complete graph with edges $\mathcal{G}_E~=~\kset{(i,j)}{1 \leq i < j \leq \sdim}$ by solving Problem~\eqref{prob:prob} with feasible set
    \begin{equation*}
        \pset = \left\{
            \pv \in \{0,1\}^{\pdim}
            \left|
            \begin{array}{ll}
                \sum_{(i,j) \in \mathcal{G}_E}\pvi{ij} = 2 & \forall i \in \mathcal{G}_V \\
                \sum_{(i,j) \in \mathcal{G}_E \cap \mathcal{S}} \pvi{ij} \leq \card{\mathcal{S}}-1 & \forall \mathcal{S} \subsetneq \mathcal{G}_V, \card{\mathcal{S}} \geq 2
            \end{array}
            \right.
        \right\}
    \end{equation*}
    with $\pdim = \card{\mathcal{G}_E}$, and where the cost vector in the domain $\costdomain = \kR_-^{\pdim}$ encodes the opposite of the distance between each pair of vertices in the graph, with negative-infinite values to model instances with infeasible connections. For this class, the number of feasible solutions scales as $\card{\pset} = \tfrac{1}{2}(\sdim - 1)!$.\\
    \item \textsc{Cut}: Cut problems seeking a partition of vertices $\mathcal{G}_V = \{1,\dots,\sdim\}$ in an undirected complete graph with edges $\mathcal{G}_E = \kset{(i,j)}{1 \leq i < j \leq \sdim}$ of minimum cut weight by solving Problem~\eqref{prob:prob} with the feasible set
    \begin{equation*}
        \pset = 
        \left\{\pv \in \{0,1\}^{\pdim}
            \left|
            \exists \mathcal{S} \subseteq \mathcal{G}_V, \ \pvi{ij} = 1 \iff (i \in \mathcal{S},j \notin \mathcal{S})~\text{or}~(i \notin \mathcal{S},j \in \mathcal{S})
            \right.
        \right\}
    \end{equation*}
    with $\pdim = \card{\mathcal{G}_E}$, and where the cost vector in the domain $\costdomain = \kR_-^{\pdim}$ encodes the opposite of the weight of each edge in the graph, with negative-infinite values to model instances with infeasible connections. For this class, the number of feasible solutions scales as $\card{\pset} = 2^{\sdim-1}$.\\
\end{itemize}
Throughout these experiments, we assume that the feasible set in Problem~\eqref{prob:prob} is known explicitly.
Unless stated otherwise, we use the sequence $\hyperplanenumber \in \{1,2,\dots,\card{\dividers(\pset)}\}$ for the iterative \gls{ldt} construction method described in \Cref{sec:dynamic:dynprog} and enable all the acceleration strategies proposed in \Cref{sec:dynamic:acceleration}.
During the construction of the normal fan $\dualfaces(\pset)$ and the computation of candidate cones and dividers in \Cref{alg:buildtree-sampling}, the feasibility of linear systems is checked using \textsc{Gurobi}~\citep{gurobi}.
The construction of $\convhull(\pset)$ is performed using \textsc{Cddlib} \citep{fukuda2003cddlib}.
For reproducibility, our implementation is available as open-source.\footnote{\texttt{\url{https://github.com/TheoGuyard/treeco}}} All computations are run on a single-core AMD 9654 CPU clocked at 2.40 GHz with 256 GB of RAM.

\subsection{Policy Structure and Complexity Characterization}
\label{sec:numerics:structure}

We first study the structural properties of two policy designs for \glspl{ilp}:
\begin{enumerate}[label=(\roman*)]
    \item A convex-hull-based policy as presented in~\citep{seidel1991small} where $\convhull(\pset)$ is built during the construction step, and Problem~\eqref{prob:prob} is solved during the evaluation step via its perfect linear relaxation $\max_{\pv \in \convhull(\pset)} \transpose{\cv}\pv$ using the simplex method.
    \item The proposed \gls{ldt} policy where a minimal-depth \gls{ldt} is built during the construction step as described in \Cref{sec:dynamic}, and Problem~\eqref{prob:prob} is solved during the evaluation step by traversing this \gls{ldt} structure.
\end{enumerate}
~
\paragraph{Geometrical analysis}
The middle part of \Cref{tab:numerics-complexity} reports geometric measures relevant to each approach. For the convex-hull policy, we report the number of vertices and facets of $\convhull(\pset)$. For the \gls{ldt} policy, we report the number of cones and dividers in $\dualfaces(\pset)$, together with the number of distinct dividers obtained after removing co-linear duplicates. These quantities provide insight into evaluation complexity. For the convex-hull policy, a larger number of vertices and facets in $\convhull(\pset)$ typically leads to more simplex pivots during evaluation. For the \gls{ldt} policy, a larger number of cones and dividers in $\dualfaces(\pset)$ induces a more intricate partition of the cost-vector space, and therefore potentially a deeper \gls{ldt} to traverse at query time. Interestingly, we observe that accounting for co-linear duplicates significantly reduces the number of dividers in $\dualfaces(\pset)$ for the \textsc{Knp} and \textsc{Tsp} instances. Only one representative of each class of co-linear dividers needs to be retained when building the \gls{ldt} policy, since such dividers carry the same information for locating the normal-fan cone containing the queried cost vector. No co-linear duplicates are observed for the \textsc{Cut} class. \\

\paragraph{Evaluation complexity}
To complement the qualitative insights provided by the structures of $\convhull(\pset)$ and $\dualfaces(\pset)$, the rightmost part of \Cref{tab:numerics-complexity} reports the maximum, average, and minimum numbers of linear tests of the form $\transpose{\mathbf{a}}\cv \leq b$ required at query time. For the convex hull policy, this corresponds to the number of possible pivots tested by the simplex method, provided that it is initialized directly at the optimal vertex of $\convhull(\pset)$. This assumption is highly optimistic in practice, and so are the metrics reported for this policy. For the \gls{ldt} policy, it corresponds to the depth of the leaf reached by the queried cost vector. Across all instances for which construction could be completed, the \gls{ldt} policy exhibits a smaller worst-case number of linear tests performed at query time than the best-case of the convex hull policy. In other words, it is guaranteed to solve Problem~\eqref{prob:prob} using fewer linear tests for \emph{every queried cost vector}, even under the optimistic assumption that the simplex method is initialized at the optimal vertex of $\convhull(\pset)$. This gives strong evidence that the proposed \gls{ldt} policy can have practical uses. This point is further validated empirically in \Cref{sec:numerics:evaluation}. Even when only a valid \gls{ldt} policy could be constructed, rather than a minimum-depth one, it still outperforms the convex hull policy in terms of the maximum, average, and minimum linear tests performed. As an example, the \gls{ldt} policy requires at most 22 linear tests to identify an optimal solution to the \textsc{Tsp}(6) instance among its 60 feasible solutions. In contrast, the convex hull policy requires at least 41 linear tests, even when initialized at an optimal vertex of $\convhull(\pset)$.\\

Beyond complexity characterization, \Cref{tab:numerics-complexity} also highlights the computational effort required to construct \gls{ldt} policies, which echoes \Cref{prop:bounded-ldt}. Within the prescribed time and memory budget, construction could not be completed for instances beyond $\textsc{Knp}(16)$, $\textsc{Cut}(7)$, and $\textsc{Tsp}(6)$. However, the convex hull policy could not be constructed beyond $\textsc{Cut}(7)$ and $\textsc{Tsp}(7)$ either, indicating that \gls{ilp} policy construction is a complicated task in general. This aspect is further investigated in \Cref{sec:numerics:construction}.

\begin{sidewaystable}[htbp]
    \centering
    \setlength{\tabcolsep}{6pt}
    \small
    \begin{tabular}{cr||rr|rr|rrr||rrr|rrr}
    \toprule
    \multicolumn{2}{c||}{Instance} &
    \multicolumn{2}{c|}{Feas. set} &
    \multicolumn{2}{c|}{Convex hull} &
    \multicolumn{3}{c||}{Normal fan} & 
    \multicolumn{3}{c|}{Simplex pivots} &
    \multicolumn{3}{c}{\gls{ldt} depth}
    \\
    Class & $\sdim$ & $\pdim$ & $\card{\pset}$ 
    & Vertices & Facets
    & Cones & Dividers & Ind. div.
    & Max & Avg & Min
    & Max & Avg & Min
    \\
    \midrule
    \textsc{Knp} & 2 & 2 & 3
    & 3 & 3
    & 3 & 3 & 3 
    & 2 & 2.00 & 2 
    & 1 & 1.00 & 1
    \\
    \textsc{Knp} & 3 & 3 & 5
    & 5 & 5
    & 5 & 8 & 6 
    & 4 & 3.20 & 3 
    & 1 & 1.00 & 1
    \\
    \textsc{Knp} & 4 & 4 & 7
    & 7 & 6
    & 7 & 15 & 10
    & 6 & 4.29 & 4 
    & 2 & 2.00 & 2
    \\
    \textsc{Knp} & 5 & 5 & 10
    & 10 & 9
    & 10 & 30 & 20
    & 9 & 6.00 & 5 
    & 4 & 4.00 & 4
    \\
    \textsc{Knp} & 6 & 6 & 14
    & 14 & 10
    & 14 & 51 & 30
    & 13 & 7.29 & 6
    & 4 & 4.00 & 4 
    \\
    \textsc{Knp} & 7 & 7 & 19
    & 19 & 15 
    & 19 & 89 & 53 
    & 18 & 9.36 & 7
    & 6 & 6.00 & 6  
    \\
    \textsc{Knp} & 8 & 8 & 25
    & 25 & 17
    & 25 & 137 & 75 
    & 24 & 10.96 & 8
    & 8 & 7.67 & 7  
    \\
    \textsc{Knp} & 9 & 9 & 33
    & 33 & 26
    & 33 & 226 & 123 
    & 32 & 13.70 & 9 
    & ${}^\ast$10 & ${}^\ast$9.52 & ${}^\ast$8 
    \\
    \textsc{Knp} & 10 & 10 & 43
    & 43 & 35
    & 43 & 339 & 176 
    & 42 & 15.76 & 10 
    & ${}^\ast$11 & ${}^\ast$10.65 & ${}^\ast$9 
    \\
    \textsc{Knp} & 11 & 11 & 55
    & 55 & 55 
    & 55 & 515 & 259 
    & 54 & 18.72 & 11 
    & ${}^\ast$13 & ${}^\ast$12.52 & ${}^\ast$11 
    \\
    \textsc{Knp} & 12 & 12 & 70
    & 70 & 64 
    & 70 & 752 & 361 
    & 69 & 21.49 & 12 
    & ${}^\ast$16 & ${}^\ast$14.63 & ${}^\ast$12 
    \\
    \textsc{Knp} & 13 & 13 & 88
    & 88 & 95 
    & 88 & 1,100 & 518 
    & 87 & 25.00 & 13 
    & ${}^\ast$18 & ${}^\ast$16.52 & ${}^\ast$12 
    \\
    \textsc{Knp} & 14 & 14 & 110
    & 110 & 130 
    & 110 & 1,558 & 707 
    & 109 & 28.33 & 14 
    & ${}^\ast$21 & ${}^\ast$18.24 & ${}^\ast$14 
    \\
    \textsc{Knp} & 15 & 15 & 137
    & 137 & 218 
    & 137 & 2,243 & 990
    & 136 & 32.74 & 15 
    & ${}^\ast$24 & ${}^\ast$20.29 & ${}^\ast$15 
    \\
    \textsc{Knp} & 16 & 16 & 169
    & 169 & 301
    & 169 & 3,102 & 1,317 
    & 168 & 36.71 & 16 
    & ${}^\ast$27 & ${}^\ast$22.41 & ${}^\ast$16 
    \\
    \textsc{Knp} & 17 & 17 & 207
    & 207 & 467
    & -- & -- & -- 
    & 206 & 41.51 & 17 
    & -- & -- & -- 
    \\
    \textsc{Knp} & 18 & 18 & 253
    & 253 & 727
    & -- & -- & -- 
    & 252 & 46.67 & 18 
    & -- & -- & -- 
    \\
    \textsc{Knp} & 19 & 19 & 307
    & 307 & 1,127 
    & -- & -- & -- 
    & 306 & 52.21 & 19 
    & -- & -- & -- 
    \\
    \textsc{Knp} & 20 & 20 & 371
    & 371 & 1,672 
    & -- & -- & -- 
    & 370 & 58.11 & 20 
    & -- & -- & -- 
    \\
    \midrule
    \textsc{Cut} & 3 & 3 & 4
    & 4 & 4 
    & 4 & 6 & 6 
    & 3 & 3.00 & 3
    & 2 & 2.00 & 2  
    \\
    \textsc{Cut} & 4 & 6 & 8
    & 8 & 7 
    & 8 & 21 & 21 
    & 7 & 6.00 & 6 
    & 6 & 4.81 & 3 
    \\
    \textsc{Cut} & 5 & 10 & 16
    & 16 & 68 
    & 16 & 105 & 105 
    & 15 & 14.00 & 14
    & 10 & 10.61 & 5 
    \\ 
    \textsc{Cut} & 6 & 15 & 32
    & 32 & 693 
    & 32 & 496 & 496 
    & 31 & 30.00 & 30
    & ${}^\ast$28 & ${}^\ast$21.87 & ${}^\ast$7  
    \\
    \textsc{Cut} & 7 & 21 & 64
    & 64 & 121,467 
    & 64 & 1,953 & 1,953
    & 63 & 62.00 & 62 
    & ${}^\ast$55 & ${}^\ast$29.03 & ${}^\ast$10 
    \\
    \textsc{Cut} & 8 & 28 & 128
    & -- & --
    & -- & -- & -- 
    & -- & -- & --
    & -- & -- & --
    \\
    \midrule
    \textsc{Tsp} & 4 & 6 & 3
    & 3 & 7
    & 3 & 3 & 3 
    & 3 & 2.00 & 2
    & 2 & 2.00 & 2  
    \\
    \textsc{Tsp} & 5 & 10 & 12
    & 12 & 25 
    & 12 & 60 & 30 
    & 12 & 10.00 & 10
    & 8 & 7.00 & 5  
    \\
    \textsc{Tsp} & 6 & 15 & 60
    & 60 & 106
    & 60 & 1,230 & 555 
    & 60 & 41.00 & 41 
    & ${}^\ast$22 & ${}^\ast$17.60 & ${}^\ast$9 
    \\
    \textsc{Tsp} & 7 & 21 & 360
    & 360 & 3,444
    & 360 & 30,240 & 9,660 
    & 360 & 168.00 & 168 
    & -- & -- & -- 
    \\
    \textsc{Tsp} & 8 & 28 & 2,520
    & -- & -- 
    & -- & -- & -- 
    & -- & -- & --
    & -- & -- & -- 
    \\
    \bottomrule
\end{tabular}
    \caption{Geometric and evaluation complexity measures for convex hull and \gls{ldt} policies. The sign `--' indicates that a policy could not be built under the 12-hour time limit. The sign `${}^*$' indicates that the iterative construction procedure for the \gls{ldt} policy did not run until completion, but still returned a policy satisfying \eqref{eq:policy-design} with no guarantees of minimal depth.}
    \label{tab:numerics-complexity}
\end{sidewaystable}

\subsection{Practical Evaluation Performance}
\label{sec:numerics:evaluation}

We now assess the practical performance of \gls{ldt} policies for recovering an optimal solution to Problem~\eqref{prob:prob}. \Cref{fig:numerics-evaluation} reports the average time required for this evaluation step and compares it with:
\begin{itemize}
    \item A straightforward brute-force approach, which evaluates all feasible solutions in $\pset$ to identify an optimal one.
    \item A direct optimization baseline tailored to each problem class: Bellman's algorithm for \textsc{Knp} \citep{toth1980dynamic}, Stoer--Wagner's algorithm for \textsc{Cut} \citep{stoer1997simple}, and Held--Karp's algorithm for \textsc{Tsp} \citep{held1970traveling}. These methods perform well on small- to moderate-sized instances, but must be run from scratch for each new cost vector.
    \item The convex hull policy considered in \Cref{sec:numerics:structure}, which solves the perfect linear relaxation of Problem~\eqref{prob:prob} over $\convhull(\pset)$ using \textsc{Gurobi}. The time required to build $\convhull(\pset)$ and instantiate the \textsc{Gurobi} model is excluded.
\end{itemize}
The reported results are averaged over 1000 cost vectors sampled uniformly at random from the unit ball, intersected with the corresponding cost domain of each \gls{ilp} class described in \Cref{sec:numerics:instances}. To obtain reliable timings and avoid biases caused by run times close to clock resolution, all methods were implemented in C/C++, and cost vectors were generated beforehand.\\

\begin{figure}[htbp]
    \centering
    \def\solvers{
    bruteforce/Brute force,
    baseline/Optim. baseline,
    cvhull/Convex hull policy,
    ldtree/LDT policy,
    ldtree-subopt/LDT policy (suboptimal)%
}

\pgfplotscreateplotcyclelist{cycle_evaluation}{
    {Red, very thick, mark=*, mark options={scale=0.33}},
    {Orange, very thick, mark=*, mark options={scale=0.33}},
    {Cerulean, very thick, mark=*, mark options={scale=0.33}},
    {Blue, very thick, mark=*, mark options={scale=0.33}},
    {Blue, thick, mark=*, mark options={scale=0.33, solid}, densely dotted},
}

\pgfplotstableread[col sep=comma]{data/knp_query.csv}\knpquery
\pgfplotstableread[col sep=comma]{data/cut_query.csv}\cutquery
\pgfplotstableread[col sep=comma]{data/tsp_query.csv}\tspquery

\begin{tikzpicture}
    \begin{groupplot}[
        group style={
            group size=3 by 1,
            horizontal sep=1.25cm,
        },
        width=4.4cm,
        height=4.4cm,
        ymode=log,
        grid=both,
        minor grid style={gray!25},
        major grid style={gray},
        cycle list name=cycle_evaluation,
    ]

        \nextgroupplot[
            title={\textsc{Knp} class},
            ylabel={Evaluation time},
            xlabel={Parameter $\sdim$},
            xmin=-1,
            xtick={0,5,10,15,20},
            ytick={0.00000001,0.000001,0.0001,0.01},
            yminorticks=false,
            yminorgrids=false,
            legend to name=legend_evaluation,
            legend columns=-1,
            legend style={/tikz/every even column/.append style={column sep=0.25cm}},
        ]

        \foreach \solver/\solverlabel in \solvers {
            \addplot table[
                x=dims,
                y=\solver
            ]{\knpquery};
            \ifthenelse{\equal{\solver}{ldtree-subopt}}{}{
                \addlegendentryexpanded{\solverlabel}
            }
        }

        \nextgroupplot[
            title={\textsc{Cut} class},
            xlabel={Parameter $\sdim$},
            xtick={3,4,5,6,7,8},
            ytick={0.00000001,0.000001,0.0001,0.01},
            yminorticks=false,
            yminorgrids=false,
        ]
        
        \foreach \solver/\solverlabel in \solvers {
            \addplot table[x=dims, y=\solver, restrict x to domain=0:8] {\cutquery};
        }

        \coordinate (top) at (rel axis cs:0.5,1);

        \nextgroupplot[
            title={\textsc{Tsp} class},
            xlabel={Parameter $\sdim$},
            xtick={3,4,5,6,7,8},
            ytick={0.00000001,0.000001,0.0001,0.01},
            yminorticks=false,
            yminorgrids=false,
        ]

        \foreach \solver/\solverlabel in \solvers {
            \addplot table[x=dims, y=\solver, restrict x to domain=0:8] {\tspquery};
        }

    \end{groupplot}

    \node[xshift=-0.5cm,yshift=1.25cm,font=\small] at (top.north) {\ref{legend_evaluation}};
\end{tikzpicture}
    \caption{Average evaluation time for solving the selected \gls{ilp} classes as a function of the parameter $\sdim \in \kN$. Evaluation times for the convex hull and \gls{ldt} policies are omitted when their construction could not be completed within the 12-hour time budget. For the \gls{ldt} policy, the dashed part of the curve indicates that the construction step returned a policy that satisfies \eqref{eq:policy-design}, but without any guarantee of minimum depth.}
    \label{fig:numerics-evaluation}
\end{figure}
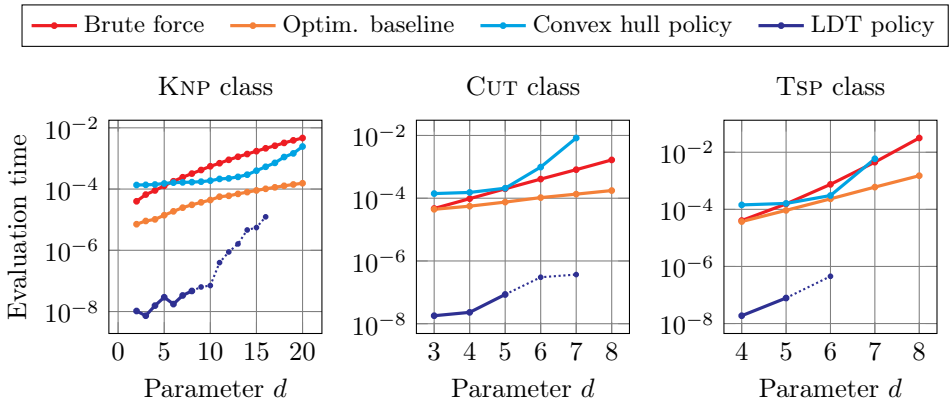

\Cref{fig:numerics-evaluation} confirms the analysis of \Cref{sec:numerics:structure}: the \gls{ldt} policy consistently outperforms both the convex hull policy and the other benchmark approaches when solving Problem~\eqref{prob:prob} for a given cost vector. Compared with the convex hull policy, it achieves speedups of up to four orders of magnitude. These gains tend to increase with problem size for \textsc{Cut} instances, in line with the observations of \Cref{sec:numerics:structure}, where $\convhull(\pset)$ exhibits a significantly more complex geometric structure than $\dualfaces(\pset)$. Comparable improvements are also observed relative to brute-force and optimization baseline methods. We further note that the dashed portion of the \gls{ldt} curve could likely be lowered with a larger construction budget, as shallower \gls{ldt} structures could have been obtained.

\subsection{Construction Performance}
\label{sec:numerics:construction}

As noted in the previous experiments, constructing \gls{ldt} policies requires substantial computational resources. In this section, we examine the factors driving this cost more closely and measure the impact of some of the acceleration strategies introduced in \Cref{sec:dynamic:acceleration} through an ablation study. Although policy construction involves computing the normal fan $\dualfaces(\pset)$, this step is negligible compared with the execution of \Cref{alg:buildtree-sampling} to synthesize the \gls{ldt} structure. It accounts for less than $0.1\%$ of the total running time across all instances. We therefore focus exclusively on metrics related to \gls{ldt} synthesis.\\

\paragraph{Construction metrics}
\Cref{tab:numerics-construction} reports several quantities related to the construction of \gls{ldt} policies via \Cref{alg:buildtree-sampling}. We consider the iterative strategy presented in \Cref{sec:dynamic} where one additional divider is introduced at each iteration, as well as a greedy strategy where the method is only run with $\hyperplanenumber = 1$, i.e., where a single divider is inspected per node. For both methods, we report the total running time of \Cref{alg:buildtree-sampling}, as well as the number of processed nodes and linear programs solved to recover the candidate cones and dividers associated with each node. We also report the depth gap between the \gls{ldt} produced by the greedy strategy and that obtained with the iterative strategy, since the greedy variant may overestimate the minimum depth by restricting the search to a single divider per node.

We observe that both the number of processed nodes and the number of linear programs solved grow rapidly with the instance dimension, reflecting the increasing memory and computational requirements of \gls{ldt} policy construction. This growth also stems from the relatively loose lower bound~\eqref{eq:lower-bound} used for pruning in \Cref{alg:buildtree-sampling}. Designing tighter bounds would likely be highly beneficial, but doing so remains challenging. The greedy strategy exhibits substantially lower running times. Interestingly, on instances where a minimum-depth \gls{ldt} could be constructed, we observe that the greedy variant consistently attains the same depth. It offers a favorable trade-off, enabling faster construction while maintaining strong evaluation performance through shallow \gls{ldt} structures. \\

\begin{table}[htbp]
    \setlength{\tabcolsep}{3.75pt}
    \centering
    \small
    \newcolumntype{E}{S[
  scientific-notation = true,
  retain-zero-exponent = true,
  round-mode = places,
  round-precision = 1,
  table-format = 1.1e2
]}

\begin{tabular}{
    cr
    |
    EEE 
    |
    EEEc
}
    \toprule
    \multicolumn{2}{c|}{Instance} &
    \multicolumn{3}{c|}{Iterative construction} &
    \multicolumn{4}{c}{Greedy construction}
    \\
    Class & $\sdim$ &
    \multicolumn{1}{l}{Time (sec.)} &
    \multicolumn{1}{l}{Nodes} &
    \multicolumn{1}{l|}{LP solved} &
    \multicolumn{1}{l}{Time (sec.)} &
    \multicolumn{1}{l}{Nodes} &
    \multicolumn{1}{l}{LP solved} &
    Gap
    \\
    \midrule
    \textsc{Knp} & 2
    & 0.001
    & 3
    & 11
    & 0.001
    & 3
    & 11
    & 0
    \\
    \textsc{Knp} & 3
    & 0.002
    & 3
    & 38
    & 0.002
    & 3
    & 38
    & 0
    \\
    \textsc{Knp} & 4
    & 0.005
    & 11
    & 106
    & 0.005
    & 7
    & 106
    & 0
    \\
    \textsc{Knp} & 5
    & 0.034
    & 165
    & 693
    & 0.024
    & 31
    & 475
    & 0
    \\
    \textsc{Knp} & 6
    & 0.056
    & 165
    & 1063
    & 0.046
    & 31
    & 859
    & 0
    \\
    \textsc{Knp} & 7
    & 2.691
    & 9367
    & 33617
    & 0.198
    & 127
    & 3068
    & 0
    \\
    \textsc{Knp} & 8
    & 597.714
    & 1403607
    & 4923709
    & 0.620
    & 383
    & 8187
    & 0
    \\
    \textsc{Knp} & 9
    & \multicolumn{1}{c}{--}
    & \multicolumn{1}{c}{--}
    & \multicolumn{1}{c|}{--}
    & 2.761
    & 1343
    & 28753
    & \multicolumn{1}{c}{$\times$} 
    \\
    \textsc{Knp} & 10
    & \multicolumn{1}{c}{--}
    & \multicolumn{1}{c}{--}
    & \multicolumn{1}{c|}{--}
    & 8.003
    & 3071
    & 62039
    & \multicolumn{1}{c}{$\times$} 
    \\
    \textsc{Knp} & 11
    & \multicolumn{1}{c}{--}
    & \multicolumn{1}{c}{--}
    & \multicolumn{1}{c|}{--}
    & 38.096
    & 11135
    & 209430
    & \multicolumn{1}{c}{$\times$} 
    \\
    \textsc{Knp} & 12
    & \multicolumn{1}{c}{--}
    & \multicolumn{1}{c}{--}
    & \multicolumn{1}{c|}{--}
    & 127.538
    & 32383
    & 589880
    & \multicolumn{1}{c}{$\times$} 
    \\
    \textsc{Knp} & 13
    & \multicolumn{1}{c}{--}
    & \multicolumn{1}{c}{--}
    & \multicolumn{1}{c|}{--}
    & 618.446
    & 117503
    & 2067505
    & \multicolumn{1}{c}{$\times$} 
    \\
    \textsc{Knp} & 14
    & \multicolumn{1}{c}{--}
    & \multicolumn{1}{c}{--}
    & \multicolumn{1}{c|}{--}
    & 2515.639
    & 388095
    & 6580522
    & \multicolumn{1}{c}{$\times$} 
    \\
    \textsc{Knp} & 15
    & \multicolumn{1}{c}{--}
    & \multicolumn{1}{c}{--}
    & \multicolumn{1}{c|}{--}
    & 11612.589
    & 1276159
    & 21195714
    & \multicolumn{1}{c}{$\times$} 
    \\
    \textsc{Knp} & 16
    & \multicolumn{1}{c}{--}
    & \multicolumn{1}{c}{--}
    & \multicolumn{1}{c|}{--}
    & 34410.410
    & 6352379
    & 98567113
    & \multicolumn{1}{c}{$\times$} 
    \\
    \textsc{Knp} & 17
    & \multicolumn{1}{c}{--}
    & \multicolumn{1}{c}{--}
    & \multicolumn{1}{c|}{--}
    & \multicolumn{1}{c}{--}
    & \multicolumn{1}{c}{--}
    & \multicolumn{1}{c}{--}
    & \multicolumn{1}{c}{--}
    \\
    \midrule
    \textsc{Cut} & 3
    & 0.001
    & 11
    & 19
    & 0.001
    & 7
    & 19
    & 0
    \\
    \textsc{Cut} & 4
    & 2.820
    & 15823
    & 45932
    & 0.042
    & 63
    & 865
    & 0
    \\
    \textsc{Cut} & 5
    & 38104.309
    & 15952401
    & 59501058
    & 2.643
    & 1151
    & 26923
    & 0
    \\ 
    \textsc{Cut} & 6
    & \multicolumn{1}{c}{--}
    & \multicolumn{1}{c}{--}
    & \multicolumn{1}{c|}{--}
    & 155.541
    & 20575
    & 662401
    & \multicolumn{1}{c}{$\times$} 
    \\
    \textsc{Cut} & 7
    & \multicolumn{1}{c}{--}
    & \multicolumn{1}{c}{--}
    & \multicolumn{1}{c|}{--}
    & 13230.398
    & 526783
    & 19597724
    & \multicolumn{1}{c}{$\times$} 
    \\
    \textsc{Cut} & 8
    & \multicolumn{1}{c}{--}
    & \multicolumn{1}{c}{--}
    & \multicolumn{1}{c|}{--}
    & \multicolumn{1}{c}{--}
    & \multicolumn{1}{c}{--}
    & \multicolumn{1}{c}{--}
    & \multicolumn{1}{c}{--}
    \\
    \midrule
    \textsc{Tsp} & 4
    & 0.001
    & 11
    & 17
    & 0.001
    & 7
    & 17
    & 0
    \\
    \textsc{Tsp} & 5
    & 1079.134
    & 3309852
    & 13380480
    & 0.291
    & 215
    & 4591
    & 0
    \\
    \textsc{Tsp} & 6
    & \multicolumn{1}{c}{--}
    & \multicolumn{1}{c}{--}
    & \multicolumn{1}{c|}{--}
    & 904.906
    & 79393
    & 3988532
    & \multicolumn{1}{c}{$\times$} 
    \\
    \textsc{Tsp} & 7
    & \multicolumn{1}{c}{--}
    & \multicolumn{1}{c}{--}
    & \multicolumn{1}{c|}{--}
    & \multicolumn{1}{c}{--}
    & \multicolumn{1}{c}{--}
    & \multicolumn{1}{c}{--}
    & \multicolumn{1}{c}{--}
    \\
    \bottomrule
\end{tabular}
    \caption{\gls{ldt} construction metrics. The symbol `--' indicates that \Cref{alg:buildtree-sampling} did not complete within the 12-hour time limit. When a minimum-depth \gls{ldt} could not be obtained, the greedy construction depth gap is reported as `$\times$'.}
    \label{tab:numerics-construction}
\end{table}

\paragraph{Ablation study}
We now assess the impact of the acceleration strategies introduced in \Cref{sec:dynamic:acceleration} through an ablation study on the \textsc{Knp}($6$), \textsc{Cut}($4$), and \textsc{Tsp}($5$) instances. \Cref{fig:numerics-ablation} reports statistics similar to those in \Cref{tab:numerics-construction}, after disabling the inference rules used for candidate cone and divider computation, or the divider-sorting procedure. The results suggest that the inference rules consistently provide substantial speedups, achieving more than an order-of-magnitude reduction in solving time. In contrast, the impact of divider sorting appears more limited, but still provides an acceleration factor of about $1.5$ for the \textsc{Tsp}($5$) instance.

\begin{table}[htbp]
    \centering
    \small
    \newcolumntype{E}{S[
  scientific-notation = true,
  retain-zero-exponent = true,
  round-mode = places,
  round-precision = 1,
  table-format = 1.1e2
]}

\begin{tabular}{cc|l|EEEc}
    \toprule
    Class & $\sdim$ & 
    Method & 
    \multicolumn{1}{l}{Time (sec.)} & 
    \multicolumn{1}{l}{Nodes} & 
    \multicolumn{1}{l}{LP solved} 
    \\
    \midrule    
    \textsc{Knp} & 7 & All accelerations & 2.691 & 9367 & 33617 
    \\
    \textsc{Knp} & 7 & No inference rules & 81.281 & 9367 & 955881 
    \\
    \textsc{Knp} & 7 & No separator sorting & 3.182 & 9367 & 33617 
    \\
    \midrule
    \textsc{Cut} & 4 & All accelerations & 2.820 & 15823 & 45932 
    \\
    \textsc{Cut} & 4 & No inference rules & 17.804 & 15823 & 252193 
    \\
    \textsc{Cut} & 4 & No separator sorting & 2.926 & 15862 & 48014 
    \\
    \midrule
    \textsc{Tsp} & 5 & All accelerations & 1079.134 & 3309852 & 13380480 
    \\
    \textsc{Tsp} & 5 & No inference rules & 14921.393 & 3309852 & 73683968 
    \\
    \textsc{Tsp} & 5 & No separator sorting & 1533.666 & 5175182 & 19699236 
    \\
    \bottomrule
\end{tabular}
    \caption{Construction statistics for different algorithmic configurations.}
    \label{fig:numerics-ablation}
\end{table}

\subsection{Toward Scalable Policies via Nearest Neighbor Search Methods}
\label{sec:numerics:nns}

To conclude our analyses, we explore the use of \Cref{prop:nns}, which establishes that \glspl{ilp} with purely binary feasible sets can be reduced to \gls{nns} problems. In particular, we investigate whether the transformed set $\scalingfunc(\pset)$ arising in this reduction can be encoded using standard \gls{nns} data structures so as to enable efficient recovery of optimal solutions to Problem~\eqref{prob:prob}. To this end, \Cref{fig:numerics-nns} reports the construction times of KD-tree \citep{bentley1975multidimensional}, HNSW \citep{malkov2018efficient}, and FAISS \citep{johnson2019billion} methods to encode the set of points in $\scalingfunc(\pset)$. We also report their evaluation time to recover a solution\footnote{While KD-tree and FAISS return an exact solution to \gls{nns} problems, HNSW is an approximate method. On average over all runs, HNSW yields a solution achieving a relative optimality gap of $5.80\%$ for \textsc{Knp} instances, $1.64\%$ for \textsc{Cut} instances, and $0.35\%$ for \textsc{Tsp} instances.} to Problem~\eqref{prob:prob}, averaged over 1000 cost vectors randomly sampled over the unit ball intersected with the respective domain of each \gls{ilp} class given in \Cref{sec:numerics:instances}. This average evaluation time is compared with those of the brute-force and optimization baseline approaches considered in \Cref{sec:numerics:evaluation}, as well as the \gls{ldt} policy. Two main observations emerge:
\begin{itemize}
    \item Although constructing \gls{ldt}-based policies is computationally demanding, \gls{nns} data structures offer substantially better scalability. They can be built for instances up to \textsc{Knp}(20) within milliseconds, \textsc{Cut}(20) within seconds, and \textsc{Tsp}(15) within hours.
    \item In terms of evaluation time, policies based on \gls{nns} data structures remain competitive. In the considered settings, HNSW achieves speedups of at least one order of magnitude on small- to moderate-sized instances relative to the optimization baseline.
\end{itemize}
Overall, these results suggest that \gls{nns}-based policies are a promising direction for binary \glspl{ilp}, offering a more favorable trade-off between construction cost and evaluation performance. Moreover, there remains substantial room to further exploit the specific structure of the set $\scalingfunc(\pset)$ induced by binary \glspl{ilp} in \gls{nns} techniques.

\begin{figure}[!ht]
    \centering
    \def\solvers{
    bruteforce/Brute force,
    baseline/Optim. baseline,
    kdtree/KD-tree,
    hnsw/HNSW,
    faiss/FAISS,
    ldtree/LDT policy,
    ldtree-subopt/LDT policy (suboptimal)%
}

\pgfplotscreateplotcyclelist{cycle_nns}{
    {Red, very thick, mark=*, mark options={scale=0.33}},
    {Orange, very thick, mark=*, mark options={scale=0.33}},
    {SeaGreen, very thick, mark=*, mark options={scale=0.33}},
    {OliveGreen, very thick, mark=*, mark options={scale=0.33}},
    {LimeGreen, very thick, mark=*, mark options={scale=0.33}},
    {Blue, very thick, mark=*, mark options={scale=0.33}},
    {Blue, thick, mark=*, mark options={scale=0.33, solid}, densely dotted},
}

\pgfplotstableread[col sep=comma]{data/knp_build.csv}\knpbuild
\pgfplotstableread[col sep=comma]{data/knp_query.csv}\knpquery
\pgfplotstableread[col sep=comma]{data/cut_build.csv}\cutbuild
\pgfplotstableread[col sep=comma]{data/cut_query.csv}\cutquery
\pgfplotstableread[col sep=comma]{data/tsp_build.csv}\tspbuild
\pgfplotstableread[col sep=comma]{data/tsp_query.csv}\tspquery

\begin{tikzpicture}
    \begin{groupplot}[
        group style={
            group size=3 by 2,
            horizontal sep=1.25cm,
            vertical sep=1.25cm
        },
        width=4.5cm,
        height=4.5cm,
        ymode=log,
        xtick={5,10,15,20},
        grid=both,
        minor grid style={gray!25},
        major grid style={gray},
        cycle list name=cycle_nns
    ]

        \nextgroupplot[
            title={\textsc{Knp} class},
            ylabel={Construction time (sec.)},
            ytick={0.00001,0.001,0.1,10,1000,100000},
        ]
        \foreach \solver/\solverlabel in \solvers {
            \addplot table[x=dims, y=\solver]{\knpbuild};
        }
        \addplot[black, very thick, dashed] table[x=dims, y=time-limit]{\knpbuild};

        \nextgroupplot[
            title={\textsc{Cut} class},
            ytick={0.00001,0.001,0.1,10,1000,100000},
        ]
        \foreach \solver/\solverlabel in \solvers {
            \addplot table[x=dims, y=\solver]{\cutbuild};
        }
        \addplot[black, very thick, dashed] table[x=dims, y=time-limit]{\cutbuild};

        \coordinate (top) at (rel axis cs:0.5,1);

        \nextgroupplot[
            title={\textsc{Tsp} class},
            ytick={0.00001,0.001,0.1,10,1000,100000},
        ]
        \foreach \solver/\solverlabel in \solvers {
            \addplot table[x=dims, y=\solver]{\tspbuild};
        }
        \addplot[black, very thick, dashed] table[x=dims, y=time-limit]{\tspbuild};

        \nextgroupplot[
            xlabel={Parameter $\sdim$},
            ylabel={Evaluation time (sec.)},
            ytick={0.00000001,0.000001,0.0001,0.01},
            legend to name=legend_nns,
            legend columns=3,
            legend cell align=left,
            legend style={/tikz/every even column/.append style={column sep=0.25cm}}
        ]
        \foreach \solver/\solverlabel in \solvers {
            \addplot table[x=dims, y=\solver]{\knpquery};
            \ifthenelse{\equal{\solver}{ldtree-subopt}}{}{
                \addlegendentryexpanded{\solverlabel}
            }
        }

        \nextgroupplot[
            xlabel={Parameter $\sdim$},
            ytick={0.00000001,0.000001,0.0001,0.01,1}
        ]
        \foreach \solver/\solverlabel in \solvers {
            \addplot table[x=dims, y=\solver]{\cutquery};
        }

        \nextgroupplot[
            xlabel={Parameter $\sdim$},
            ytick={0.00000001,0.00001,0.01,10,10000},
            xmax = 16
        ]
        \foreach \solver/\solverlabel in \solvers {
            \addplot table[x=dims, y=\solver]{\tspquery};
        }
    \end{groupplot}

    \node[yshift=1.5cm,font=\small] at (top.north) {\ref{legend_nns}};
\end{tikzpicture}
    \caption{Construction and evaluation time for policies based on \gls{nns} data structures. The black dashed line indicates the 12-hour time budget allowed for their construction.}
    \label{fig:numerics-nns}
\end{figure}
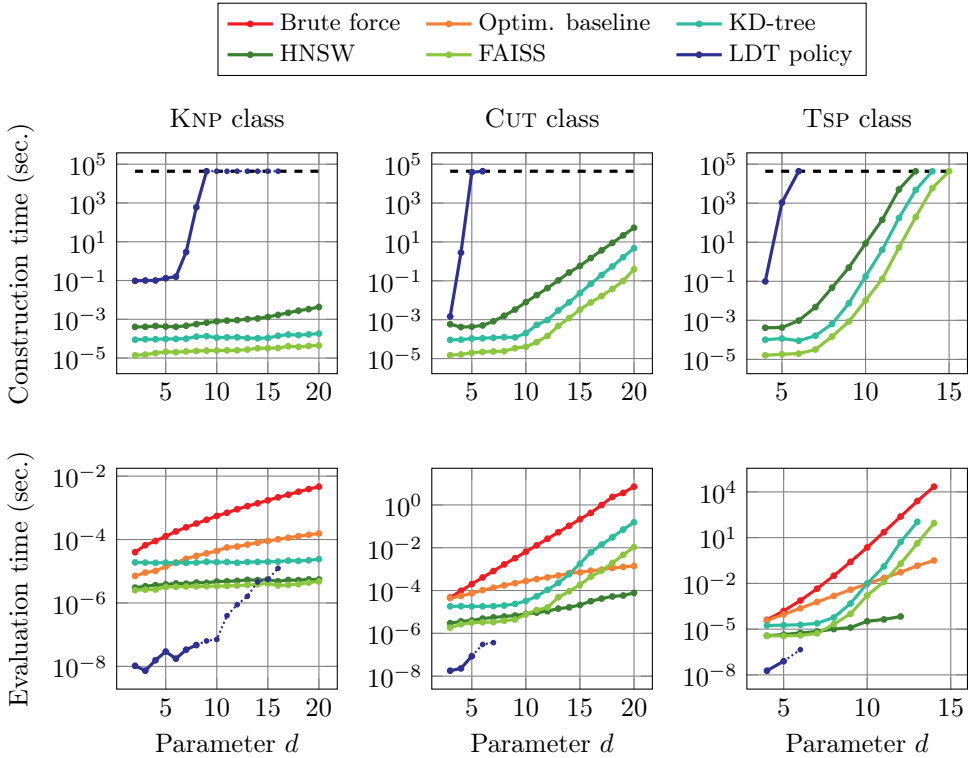

\section{Related Works on Decision Policies}
\label{sec:related-works}

Relatively few studies have considered decision policies that fit the framework of \eqref{eq:policy-design}, in which an optimal solution to Problem~\eqref{prob:prob} must be returned at query time, although this property may be crucial for high-stakes applications. Beyond \gls{ldt} policies, the main exact approaches that have been investigated are based on convex-hull representations and separation oracles. A broader literature has also studied heuristic policies, which allow approximate solutions, as well as optimization methods based on learning that harness knowledge gained on different problems.\\

\paragraph{Policies based on the convex hull}
Problem~\eqref{prob:prob} can be equivalently formulated as a linear program over the convex hull of the feasible set. Based on this observation, prior work has proposed policies in which $\convhull(\pset)$ is constructed offline during the construction phase, and a solution to Problem~\eqref{prob:prob} is recovered at query time by solving the surrogate problem $\max_{\pv \in \convhull(\pset)} \transpose{\cv}\pv$ using linear programming techniques \citep{chazelle1991optimal}. This policy design is broadly applicable, since convex hulls can be constructed generically \citep{motzkin1953double}, but their geometry is often highly complex, making the construction phase computationally challenging \citep{christofsmapo}. Moreover, when evaluation relies on the simplex algorithm, the number of pivot steps depends on the maximum vertex degree in the convex hull \citep{klee1972good}. This quantity can be exponential in $\pdim$ for some \gls{ilp} instances \citep{nikolaev2023cone}, so the evaluation complexity of such convex hull policies may itself remain exponential.\\

\paragraph{Policies based on a separation oracle}
Alternatively, instead of building $\convhull(\pset)$ explicitly, one may rely on a separation oracle for the convex hull.
If such an oracle can be built during a construction step, the ellipsoid method can be adapted to perform the evaluation step and recover an optimal solution to Problem~\eqref{prob:prob} for any queried cost vector \citep{grotschel1981ellipsoid}.
Each iteration is carried out with complexity $\bigO(\pdim^2 + \Gamma(\pdim))$, where $\Gamma(\pdim)$ denotes the complexity of querying the separation oracle. In total, $\bigO(\pdim^2 \log(\epsilon^{-1}))$ iterations are required to reach $\epsilon$-optimality during this evaluation step \citep{khachiyan1980polynomial}.
Consequently, the complexity of solving Problem~\eqref{prob:prob} to some prescribed numerical tolerance using this policy design is polynomial whenever $\Gamma(\pdim)$ is polynomial in $\pdim$. 
This property can be used to prove that maximum weighted matching problems can be solved with polynomial complexity without resorting to Edmonds' algorithm \citep{grotschel1981ellipsoid}.\\

\paragraph{Heuristic policies}
In addition to policies that must return an exact solution to Problem~\eqref{prob:prob} at query time, some studies have considered policies that allow inexact solutions. These policies typically rely on learning-based strategies during their construction phase \citep{khalil2017learning,bello2016neural}. At query time, they may either produce a surrogate problem formulation that can be solved more efficiently \citep{bertsimas2021voice,khalil2017learning}, or directly infer a solution through a machine learning model \citep{shen2023adaptive,xin2021multi}. As they do not guarantee the optimality of the solution returned, techniques to analyze the solution quality of these heuristic policies have been proposed \citep{chen2024compact}. This heuristic paradigm also allows the design of policies for more general optimization problems than \glspl{ilp} \citep{bertsimas2018binary}.\\

\paragraph{Direct optimization methods}
Beyond pure decision policies that rely on a heavy construction step to support efficient queries, the structure of specific feasible sets can also be exploited in direct optimization methods such as branch-and-cut algorithms \citep{mitchell2002branch}, decomposition techniques \citep{barnhart1998column}, or dynamic programming procedures \citep{bellman1962dynamic}. As with decision policies, a preprocessing step can be performed to leverage similarities across instances that are solved repeatedly in direct methods. For example, branching rules \citep{lodi2017learning}, exploration strategies \citep{he2014learning}, and column-generation techniques \citep{shen2022enhancing} in branch-and-bound algorithms can be implemented using trained learning models rather than prescribed rules. Heuristic methods can also benefit from such techniques to better drive their search among feasible solutions \citep{mirshekarian2018machine,manchanda2022generalization}. However, in contrast to decision policies, the complexity of solving \glspl{ilp} with these approaches remains largely tied to that of the direct optimization method considered.

\section{Conclusion and Perspectives}
\label{sec:conclusion}

In this paper, we introduced \gls{ldt} policies as an exact offline-online paradigm for \glspl{ilp} with a fixed feasible set and varying cost vectors.
We showed that, under a mild encoding assumption, such policies always exist with polynomial query complexity in the \gls{ldt} model, but that deciding whether an exact policy with a prescribed maximum number of leaves exists is $\Sigma_2^p$-complete.
Additionally, we developed a concrete synthesis method based on dynamic programming and pruning for a structured subclass of \glspl{ldt} induced by the normal fan of the feasible set. 
As shown by our numerical experiments, although policy construction can be demanding, the resulting policies are of practical interest, as they can answer repeated optimization queries orders of magnitude faster than classical exact methods on small instances. In the binary case, we also identified a natural reformulation as a nearest-neighbor search problem, which opens additional perspectives for scalable policy representations. Overall, this work proposes a new perspective on exact policy synthesis for repeated combinatorial optimization and opens several promising research directions.\\

From a theoretical standpoint, a key open question is whether the existential polynomial query-complexity guarantee can be matched by a practical construction procedure. The main challenge is that the existence result imposes no restriction on the linear tests allowed in the associated \gls{ldt}, whereas the constructive method developed in \Cref{sec:dynamic} searches only within a structured subclass. If one could show that a finite family of candidate tests is sufficient to recover the complexity bound of \Cref{prop:query-complexity}, then the practical framework introduced here could in principle be extended to synthesize such policies with polynomial query guarantees. A natural intuition is that the dividers $\dividers(\pset)$ defined in \eqref{eq:dividers} already contain all the relevant information, but establishing this remains an open question.\\

From a methodological standpoint, an important direction for future research is to improve the scalability of policy construction. In the binary case, reformulating \glspl{ilp} as \gls{nns} problems warrants further exploration. While \gls{nns} is itself challenging in high dimensions, the queries considered here possess a specific structure, since they are performed over a linear transformation of the feasible set of the underlying \gls{ilp}. This opens the door to tailored \gls{nns} methods. For example, the structure of specific \gls{ilp} instances could be leveraged during the construction of KD-trees or related geometric search methods. More broadly, one may also move beyond the strict goal of exact policy synthesis and investigate incremental constructions of \gls{ldt}- or \gls{nns}-based policies with certified approximation guarantees. Such an approach would lead to a disciplined family of approximation algorithms that explicitly trade oracle-construction effort for a controlled solution gap.\\

Finally, these policy representations may find practical value in broader optimization pipelines where repeated exact solution of small combinatorial subproblems is a computational bottleneck. 
A first class of applications arises in bilevel and hierarchical optimization, where the lower-level problem has the form considered in this paper and must be evaluated repeatedly for varying upper-level decisions.
In such settings, \gls{ldt} policies could in principle be encoded through MILP formulations of the underlying tree structure (for example, following ideas similar to \citep{Parmentier2021c}), thereby providing an alternative way to represent follower reactions alongside more classical reformulations based on optimality conditions when these are available \citep{kleinert2021survey}. A second class of applications arises in decomposable combinatorial problems. For example, cluster-first route-second approaches for vehicle routing repeatedly evaluate customer-to-vehicle assignments by solving many smaller traveling salesman subproblems \citep{Toffolo2019,Vidal2020}. When these subproblems remain within the tractable regime identified in this paper, an \gls{ldt} policy oracle could substantially reduce the overall method's computational burden and enable very fast exact cost evaluations for any given visit cluster. Finally, since any \gls{ldt} can be flattened into a sequence of nested \texttt{if-else} statements and elementary arithmetic operations, as illustrated in \Cref{app:flattened-ldt}, these policies may also provide a useful basis for designing specialized hardware for very specific control tasks, in domains where reliability and response time are critical.

\clearpage
\appendix

\section{Flattened Linear Decision Tree}
\label{app:flattened-ldt}

Any \gls{ldt} structure can be flattened into a sequence of nested \texttt{if-else} statements and elementary arithmetic operations, as illustrated in \Cref{fig:flattened-ldt}. 
The flattened representation of \gls{ldt} policies generated for our numerical experiments are available at:
\begin{center}
    ~\\
    \texttt{\url{https://github.com/TheoGuyard/ilp-oracles}}
    ~\\~\\
\end{center}
This repository is intended to serve as a baseline for a broader collaborative library that gathers optimal policies for canonical \gls{ilp} families.

\begin{figure}[htb]
    \centering
    \begin{lstlisting}[style=CStyle]
#define N 6   // problem dimension

static const int *query(double *c) {
    if (c[1] - c[3] + c[5] < 0.0) {
        if (c[2] - c[4] + c[5] < 0.0) {
            if (c[0] - c[1] - c[2] < 0.0) {
                static int x[N] = {0, 1, 1, 1, 1, 0};
                return x;
            } else {
                static int x[N] = {1, 0, 0, 1, 1, 0};
                return x;
            }
        } else {
            if (c[0] - c[1] + c[2] < 0.0) {
                if (c[2] + c[4] - c[5] < 0.0) {
                    static int x[N] = {0, 1, 0, 1, 0, 1};
                    return x;
                } else {
                    static int x[N] = {0, 1, 1, 1, 1, 0};
                    return x;
                }
            } else {
                if (c[0] - c[1] - c[4] + c[5] < 0.0) {
                    static int x[N] = {0, 1, 1, 1, 1, 0};
                    return x;
                } else {
                    static int x[N] = {1, 0, 1, 1, 0, 1};
                    return x;
                }
            }
        }
    } else {
        ...
    }
}
    \end{lstlisting}
    \caption{Flattened C representation of an \gls{ldt} policy to solve any \textsc{Cut}(4) instance, given a cost vector $\cv = (c_{12}, c_{13}, c_{14}, c_{23}, c_{24}, c_{34})$ encoding the negative of the edge weights between each pair of vertices in the graph. The total function spans 111 lines and only the 35 first ones are displayed.}
    \label{fig:flattened-ldt}
\end{figure}

\bibliographystyle{siamplain}
\bibliography{main}

\end{document}